\documentclass[12pt]{iopart}

\usepackage{iopams}
\expandafter\let\csname equation*\endcsname\relax

\expandafter\let\csname endequation*\endcsname\relax

\usepackage{amsmath}

\usepackage{epsfig}
\usepackage{graphicx}
\usepackage{latexsym}

\usepackage{textcomp}
\usepackage{color,bm}
\usepackage{multirow}
\usepackage{verbdef}
\usepackage{mathrsfs}

\usepackage{comment}

\usepackage{bm}

\usepackage{epstopdf}
\usepackage{cite}
\usepackage{float}
\usepackage{tikz}
\usepackage{booktabs}
\usepackage{amssymb}
\usepackage{algorithm}

\usetikzlibrary{arrows,backgrounds,positioning,arrows.meta,chains,decorations.pathreplacing}

\usepackage{amsthm}
\usepackage{soul}

\newtheorem{theorem}{Theorem}[section]

\newtheorem{remark}{Remark}[section]
\newtheorem{proposition}{Proposition}[section]

\newtheorem{definition}{Definition}[section]

\newtheorem{Problem}[theorem]{Problem}

\newcommand{\mD}{\mathsf{D}}
\newcommand{\mK}{\mathsf{K}}
\newcommand{\mF}{\mathsf{F}}
\newcommand{\mT}{\mathsf{T}}
\newcommand{\mM}{\mathsf{M}}
\newcommand{\mE}{\mathsf{E}}

\newcommand{\mV}{\mathsf{V}}
\newcommand{\mW}{\mathsf{W}}

\begin{document}

\title[Feature augmentation for inverse problems]{Feature augmentation for the inversion of the Fourier transform with limited data}

\author{Emma Perracchione$^{1}$, Anna Maria Massone$^{1,2}$, Michele Piana$^{1,2}$}

\address{$^1$Dipartimento di Matematica,  Universit\`{a} degli Studi di Genova, Via   Dodecaneso 35, 16146 Genova, Italy\\
$^2$CNR-SPIN, Via Dodecaneso 33, 16146 Genova, Italy\\
}
\ead{perracchione@dima.unige.it; massone@dima.unige.it; piana@dima.unige.it}

\vspace{10pt}
\begin{indented}
\item[]February 2021
\end{indented}

\begin{abstract}
We investigate an interpolation/extrapolation method that, given scattered observations of the Fourier transform, approximates its inverse. The interpolation algorithm takes advantage of modelling the available data via a shape-driven interpolation based on Variably Scaled Kernels (VSKs), whose implementation is here tailored for inverse problems. The so-constructed interpolants are used as inputs for a standard iterative inversion scheme. After providing theoretical results concerning the spectrum of the VSK collocation matrix, we test the method on astrophysical imaging benchmarks.
\end{abstract}

%
\noindent{\it Keywords}: Shape-driven interpolation, Variably Scaled Kernels, feature augmentation, Solar X-ray imaging\\
%
\submitto{\IP}

\section{Introduction}

The inversion of the Fourier transform with limited data is a well-known issue that leads to bandlimited extrapolation problems and is common to many applied sciences, such as microscopy, medical imaging, seismology, and radio-astronomy (see e.g. \cite{bertero1987super,mcgibney1993quantitative,Duijndam,richard2017interferometry}). In some cases, e.g. when the sampling is not uniform, preliminary interpolation methods are needed to reconstruct the frequency information. In view of this, we drive our attention towards  interpolation/extrapolation approaches (refer, e.g. to \cite{Massone1,sacchi}) and, specifically, to the reconstruction scheme that consists of the following two steps:
\begin{itemize}
    \item Interpolation of the scattered observations of the Fourier transform so that a uniform sampling in the frequency domain is generated. 
    \item Extrapolation and inversion of the so-generated interpolants via Fast Fourier Transform (FFT)-based iterative methods.
\end{itemize}

The extrapolation issue is typically addressed via  standard soft-thresholding approaches \cite{Daubechies}, whose performances strongly depend on the accuracy of the interpolation scheme. For multivariate (scattered) data approximation problems, many methods have already been developed; to mention a few: splines, kernel-based methods and multivariate polynomial approximation of total degree  \cite{Fasshauer15,trefethen2013approximation,schumaker_2007}. 

In the following we focus on kernel-based approximation for interpolation, which leads to intuitive collocation schemes that only depend on the distances among the scattered nodes. This eventually allows the definition of easy-to-implement feature augmentation strategies. More precisely, while interpolating functions characterized by steep gradients and/or oscillations, classical kernel-based approximation might suffer from both Gibbs and Runge phenomena \cite{Gottlieb97onthe,BLSE:RN017988920}. Such a drawback becomes even worse if one only disposes of few scattered observations, as in many applications. To overcome this limitation, we exploit a novel and promising method based on Variably Scaled Kernels (VSKs). Such kernels, introduced in \cite{Bozzini1} mainly for stability purposes, have been proven to be effective also for pattern recognition in images characterized by steep gradients \cite{vskjump,vskmpi}. 
Their peculiarity lies in the fact that they encode into the kernel itself some {\em{a priori}} information (when available) leading to a shape-driven approximation. Such additional knowledge is implicitly put into the kernel via a \emph{scaling function}, which determines the accuracy of the approximation process. 

The main goal of this study is to tailor the VSKs for the Fourier inversion issue with limited data. Precisely, given a first and possibly rough approximation of the inverse problem, we solve the forward issue and starting from the so-generated approximation in the frequency domain, we define the scaling function. In this way we provide the interpolation/extrapolation algorithm with the resources to properly compute the inverse of the Fourier transform by limited data.

We further point out that also the kernel basis needs to be properly selected for the inversion procedure, which is known to be ill-conditioned. Therefore, for the practical implementation of the VSK setting, we consider the Mat\'ern $C^0$ kernel (see e.g. \cite{Matern,Stein}) that, when dealing with real and possibly noisy data, takes advantage of its low regularity, leading to reliable results. For that kernel, we provide an analysis of its spectrum in the VSK framework. This theoretical study shows that the Mat\'ern $C^0$ VSK might be less affected by ill-conditioning than the classical kernel. 

After providing the theoretical validation for the use of VSKs in the context of inverse problems, the method is extensively tested on real and simulated data coming from the framework of astronomical imaging. We consider samples from both the {\em{Reuven Ramaty High Energy Solar Spectroscopic Imager (RHESSI)}} \cite{enlighten1658}
and the {\em{Spectrometer/Telescope for Imaging X-rays (STIX)}} \cite{refId0}, which are two telescopes recording X-rays from the Sun with the main purpose of observing solar flares. For both systems, the imaging problem can be described as an inversion of the Fourier transform with undersampled data. The results on solar flares reconstructions empirically demonstrate that VSKs are the key ingredient for robust interpolation/extrapolation algorithms. 

The plan of the paper is as follows. In Section \ref{Extrapolation} we introduce the imaging problem and we briefly review the inversion step based on FFT routines and iterative methods.  Section \ref{interp} deals with the novel approximation algorithm based on VSKs. In Section \ref{Application} we illustrate a motivating example and investigate the use of the VSKs for {\em{RHESSI}} and {\em{STIX}} visibilities. Extensive numerical experiments are carried out in Section \ref{experiments}. Our conclusions are offered in Section \ref{conclusions}.

\section{Preliminaries}
\label{Extrapolation}

The main objective of this work is to address inverse problems involving the Fourier transform via VSKs. In this direction we consider the following inversion issue. 

\begin{Problem}\label{pr1}
Let ${\cal F}: L_1 \cap L_2 (\mathbb{R}^d,\mathbb{R}) \longrightarrow {\cal F} (L_1 \cap L_2 (\mathbb{R}^d,\mathbb{R})) \subseteq L_1 \cap L_2 (\mathbb{R}^d,\mathbb{C})$, with $d \in \mathbb{N}^+$, be the Fourier operator defined as 
\begin{equation}\label{eq:00}
    ({\cal F} I)(\boldsymbol{x}) = V(\boldsymbol{u}), \quad \boldsymbol{x} \in X, \hskip 0.2cm \boldsymbol{u} \in U,
\end{equation}
where $I:X\subseteq \mathbb{R}^d \longrightarrow \mathbb{R}$ and $V:U\subseteq \mathbb{R}^d \longrightarrow \mathbb{C}$. Given some scattered observations on a compact set $D \subseteq U$ of the function $V$, the problem consists in finding an approximation of $I$.
\end{Problem}

Since Problem \ref{pr1} is common to many real imaging problems, we restrict to the cases for which $I$ is a real-valued function. To approximate Problem \ref{pr1}, we propose an approach that involves kernel-based interpolation schemes at spatial frequencies and the projected Landweber iterative method \cite{Piana_1997}. The latter is a well-known constrained iterative scheme that realizes extrapolation and artifact reduction using a soft-thresholding inversion process. For the interpolation issue, which is the focus of the present paper, we study meshless schemes that make use of feature augmentation strategies. Those methods will be analyzed in Section \ref{interp}, while here we briefly review the iterative inversion approach. 

\subsection{Extrapolation}

Given the set of scattered observations on $D\subseteq \mathbb{R}^d$, the interpolation operator returns an approximation of the complex-valued function $V$ on $U \supseteq D$. Then, for stability purposes of the inversion scheme we might not allow extrapolation via interpolation out of the compact set $D$. Therefore, the interpolant, denoted by $P_V = (P_{{\rm Re}({V})},P_{{\rm Im}({V})})$, might be first projected as follows: 
\begin{equation*}
    {P_V}^D :=  \chi_{D} P_V,
\end{equation*}
where $ \chi_{D}$ is the characteristic function operator of the compact set $D$. 

Note that, 
\begin{equation*}
    {\tilde I} =  {\cal F}^{-1}\left({P_V}^D\right),
\end{equation*}
is a noisy bandlimited approximation of the distribution ${I}:X \longrightarrow \mathbb{R}$.
An extrapolation out of the band $D$ can be computed via the projected Landweber iteration in following steps:
\begin{enumerate}
\item Initialize the iterative scheme as ${I}^{(0)} = 0$.
\item For a given threshold $\tau \in \mathbb{R}$ and starting from $I^{(k)}$, compute
\begin{equation}\label{projection-3}
    {\cal F}({I}^{(k+1)}) =  \tau {\cal F}({\tilde I} ) + (1- \tau  \chi_{D}) {\cal F}({I}^{(k)}), \quad k=1,2,\ldots.
\end{equation}
and apply the non-negativity constraint by means of the projection
\begin{equation}\label{projection-4}
I^{(k+1)} = {\cal{P}}_+I^{(k+1)}
\end{equation}
where
\begin{equation*}
    \left({\cal{P}}_+ I^{(k+1)} \right) (\boldsymbol{x})  = \left \{ 
    \begin{array}{cc}
         0 &   I^{(k+1)}(\boldsymbol{x})<0 ,\\
        I^{(k+1)}(\boldsymbol{x}) & \textrm{otherwise}~.
    \end{array} \right.
\end{equation*}
\end{enumerate}

\begin{remark}
It is a well-established result \cite{piana1996regularized} that the projection onto the convex set of non-negative functions is a way to realize extrapolation.
\end{remark}

\begin{remark}
The computation of $I^{(k+1)}$ at the right hand side of equation (\ref{projection-4}) is realized by applying an FFT-based routine to equation (\ref{projection-3}).
\end{remark}

The result of the extrapolation scheme depends on the approximation of $V:U \subseteq \mathbb{R}^d \longrightarrow \mathbb{C}$. For smooth functions, many approximation methods could be successfully implemented. However, such functions might be sampled at few scattered data and might be characterized by steep gradients. In that case we need to drive our attention towards data-driven interpolation methods.

\section{Interpolation} 
\label{interp}

Given a set scattered nodes ${\cal U}= \{\boldsymbol{u}_i=(u_i,v_i), \hskip 0.2cm i=1,\ldots,n\} \subseteq U$ and the set of associated function values ${\cal V} = \{{\boldsymbol{V}}_i=({\rm Re}({\boldsymbol{V}}_i),{\rm Im}({\boldsymbol{V}}_i)), \hskip 0.2cm i=1,\ldots,n\} \subseteq \mathbb{C}$, the aim is to find two interpolating functions $P_{{\rm Re}({V})}$ and $P_{{\rm Im}({V})}: U \longrightarrow \mathbb{R}$ so that 
\begin{equation}\label{interpolationproblem}
P_{{\rm Re}({V})}(\boldsymbol{u}_i)= {\rm Re}({\boldsymbol{V}}_i), \quad {\rm and} \quad P_{{\rm Im}({V})}(\boldsymbol{u}_i)= {\rm Im}({\boldsymbol{V}}_i), \quad i=1,\ldots,n.
\end{equation}	

For constructing the interpolants we consider kernel methods that take advantage of being meshless and naturally multivariate. 

\subsection{Review of kernel-based interpolation}

The interpolation problems \eqref{interpolationproblem} have a unique solution if $P_{{\rm Re}({V})}$ and $P_{{\rm Im}({V})} \in \textrm{span} \{ K_{\varepsilon}(\cdot,\boldsymbol{u}_i), \boldsymbol{u}_i \in {\cal U}\}$, where $K_{\varepsilon} :  U \times  U \longrightarrow \mathbb{R}$ is a strictly  positive definite and radial kernel and $\varepsilon>0$ is the so-called \emph{shape parameter}. We also remark that to the radial kernel $K$ we can associate a continuous function $\varphi_{\varepsilon}: [0,+\infty)\longrightarrow \mathbb{R}$,  such that
\begin{equation*}
K_{\varepsilon}(\boldsymbol{w},\boldsymbol{z})=\varphi_{\varepsilon}(\|\boldsymbol{w}-\boldsymbol{z}\|_2),
\end{equation*}
for all $\boldsymbol{w}, \boldsymbol{z} \in U$. The function $\varphi_{\varepsilon}$ is usually referred to as a Radial Basis Function (RBF). 

For  $\boldsymbol{u}\in U$ the interpolants  are of the form 
\begin{equation*}
P_{{\rm Re}({V})}(\boldsymbol{u}) = \sum_{k = 1}^n \alpha_k K_{\varepsilon}(\boldsymbol{u}, \boldsymbol{u}_k), \quad {\rm and} \quad P_{{\rm Im}({V})}(\boldsymbol{u}) = \sum_{k = 1}^n \beta_k K_{\varepsilon}(\boldsymbol{u}, \boldsymbol{u}_k).
\end{equation*}
The coefficients $\boldsymbol{\alpha}=(\alpha_1, \ldots, \alpha_n)^{\intercal}$ and $\boldsymbol{\beta}=(\beta_1, \ldots, \beta_n)^{\intercal}$  are determined by solving the following linear systems: 
\begin{equation}\label{lin_sys}
    \mK {\boldsymbol{\alpha}} = {\rm Re}({\boldsymbol{V}}), \quad {\rm and}  \quad \mK {\boldsymbol{\beta}} = {\rm Im}({\boldsymbol{V}}),
\end{equation}
where the entries of $\mK \in \mathbb{R}^{n \times n}$ are given by \begin{equation}\label{kernel}
\mK_{ik}= K_{\varepsilon} (\boldsymbol{u}_i , \boldsymbol{u}_k), \quad i,k=1, \ldots, n, 
\end{equation}
moreover, ${\rm Re}({\boldsymbol{V}})= ({\rm Re}({\boldsymbol{V}}_1), \ldots, {\rm Re}({\boldsymbol{V}}_n))^{\intercal}$ and ${\rm Im}({\boldsymbol{V}})= ({\rm Im}({\boldsymbol{V}}_1), \ldots, {\rm Im}({\boldsymbol{V}}_n))^{\intercal}$ are the vectors of data values.

The kernels are characterized by different regularities and especially for infinitely smooth kernels the selection of the shape parameter affects the accuracy of the reconstruction, meaning that inappropriate choices of its value might lead to poor approximations (see e.g. \cite{Fornberg,Fornberg1} for a general overview). To select safe values of the shape parameter, we need to introduce the so-called \emph{native spaces}. First, we introduce the pre-Hilbert space $H_{K_{\varepsilon}}(U)$ with reproducing kernel $K_{\varepsilon}$, given by
\begin{equation*}
H_{K_{\varepsilon}}(U)=\textrm{span} \{ K_{\varepsilon}(\cdot,\boldsymbol{u}),\;\boldsymbol{u} \in U\},
\end{equation*}
and equipped with a bilinear form $\left(\cdot,\cdot\right)_{H_{K_{\varepsilon}}(U)}$. Then, the native space ${\cal N}_{K_{\varepsilon}} (U)$ is defined as the completion of $H_{K_{\varepsilon}}(U)$ with respect to the norm $||\cdot||_{H_{K_{\varepsilon}}(U)}$, i.e. $||\nu||_{H_{K_{\varepsilon}}(U)} = ||\nu||_{{\cal N}_{K{\varepsilon}}(U)},$ for all $\nu \in H_{K_{\varepsilon}}(U)$ \cite{Fasshauer,Wendland05}.

To provide error bounds, we need to introduce one more ingredient, the so-called {\it power function}. Let $\mK$ be the interpolation matrix related to the set of nodes ${\cal U}$ and $\tilde{\mK}$ be the matrix associated to the augmented set $ \{\boldsymbol{u}\}\cup {\cal U},\; \boldsymbol{u} \in U$. The power function is defined as \cite{SDW}
\begin{equation*}
{\cal P}_{K_\varepsilon,{\cal U}} (\boldsymbol{u}):=\sqrt{\dfrac{\text{det}(\tilde{\mK} )}{ \text{det}({\mK})}}, \quad  \boldsymbol{u} \in U.
\end{equation*}

The following pointwise error bound, that uses the power function and the norm of the sought function in the native space, holds true (see e.g. \cite[Th. 14.2, p. 117]{Fasshauer}).
\begin{theorem}\label{powfuntheor}
	Let $K_{\varepsilon}: U \times U \longrightarrow \mathbb{R}$ be a strictly positive definite and continuous kernel and ${\cal U}  = \{\boldsymbol{u}_i, i=1, \ldots, n \} \subseteq U$ a set of distinct points. For all functions $ \nu :U \longrightarrow \mathbb{R}$, so that $\nu \in {\cal N}_{K_{\varepsilon}}(U)$, we have that
	\begin{equation*}
	|{\nu}\left(\boldsymbol{u}\right)-P_{\nu}\left(\boldsymbol{u}\right)| \leq {\cal P}_{K_\varepsilon,{\cal U}} (\boldsymbol{u})||\nu||_{{\cal N}_{K_{\varepsilon}}(U)}, \quad \boldsymbol{u} \in U.
	\end{equation*}
\end{theorem}
Note that Theorem \ref{powfuntheor} bounds the point-wise error in terms of the power
function that depends on the kernel and on the data points but is independent of the function values. This suggests a criterion to select a \emph{safe} shape parameter. Indeed, we will select the value of the shape parameter that minimizes the power function computed over the data. 

\begin{remark}
As we will point out via numerical experiments, this standard kernel-based interpolation procedure might suffer when interpolating functions sampled at very few points and/or characterized by steep gradients. To partially avoid this drawback we introduce the VSKs with the scope of encoding into the kernel itself some prior knowledge for the interpolation issue. 
\end{remark}

\subsection{Variably scaled kernels}

The VSKs have been introduced in \cite{Bozzini1} and later they have been studied to preserve shape properties of the reconstruction surfaces or as edge detection strategies; we refer the reader to \cite{vskjump,vskmpi,romani,rossini} for further details on the topic. Here, we investigate them in the context of inverse problems.

\subsubsection{Definition of the VSK interpolants.}

The definition of VSKs relies upon a \emph{scaling function} $\Psi: U \longrightarrow \Sigma$, where $\Sigma \subseteq \mathbb{R}^m$ and $m \geq 1$ (see  \cite[Definition 2.1, p. 4]{vskmpi}).

\begin{definition}
	Let $K: (U \times \Sigma) \times (U \times \Sigma) \longrightarrow \mathbb{R}$ be
	a continuous strictly positive definite kernel. Given a scaling function $\Psi: U \longrightarrow \Sigma$,
	 a variably scaled kernel $K^{\Psi}_{\varepsilon} : U \times U \longrightarrow \mathbb{R}$ is defined as
	\begin{equation*}
	K^{\Psi}_{\varepsilon}(\boldsymbol{{w}},\boldsymbol{{z}}):= K_{\varepsilon}((\boldsymbol{{w}},\Psi(\boldsymbol{w})),(\boldsymbol{z},\Psi(\boldsymbol{z})),
	\end{equation*}
	for $\boldsymbol{w},\boldsymbol{z}\in U$.
\end{definition}
In other words, the VSKs lead to an interpolation problem with new features defined via the function $\Psi$, thus realizing a feature augmentation process. 

Then for  $\boldsymbol{u}\in U$, in the varying scale setting, our interpolants are given by 
\begin{align*}
P^{\Psi}_{{\rm Re}({V})}(\boldsymbol{u}) & = \sum_{k = 1}^n \alpha^{\Psi}_k K^{\Psi}_{\varepsilon}(\boldsymbol{u}, \boldsymbol{u}_k),\\ & = \sum_{k = 1}^n \alpha^{\Psi}_k K_{\varepsilon}((\boldsymbol{{u}},\Psi(\boldsymbol{u})),(\boldsymbol{u}_k,\Psi(\boldsymbol{u}_k)), \\ &= P_{{\rm Re}({V})}((\boldsymbol{{u}},\Psi(\boldsymbol{u})),
\end{align*}
and 
\begin{align*}
P^{\Psi}_{{\rm Im}({V})}(\boldsymbol{u}) & = \sum_{k = 1}^n \beta^{\Psi}_k K^{\Psi}_{\varepsilon}(\boldsymbol{u}, \boldsymbol{u}_k) , \\ &=  \sum_{k = 1}^n \beta^{\Psi}_k K_{\varepsilon}((\boldsymbol{{u}},\Psi(\boldsymbol{u})),(\boldsymbol{u}_k,\Psi(\boldsymbol{u}_k)) , \\ &=  P_{{\rm Im}({V})}((\boldsymbol{{u}},\Psi(\boldsymbol{u})).
\end{align*}

The coefficients $\boldsymbol{\alpha}^{\Psi}=(\alpha^{\Psi}_1, \ldots, \alpha^{\Psi}_n)^{\intercal}$ and $\boldsymbol{\beta}^{\Psi}=(\beta^{\Psi}_1, \ldots, \beta^{\Psi}_n)^{\intercal}$  are determined by solving the two linear systems defined in \eqref{lin_sys}, with kernel matrix $\mK^{\Psi} \in \mathbb{R}^{n \times n}$, whose entries are given by 
\begin{equation}\label{kernel1}
\mK^{\Psi}_{ik}= K_{\varepsilon} ((\boldsymbol{{u}_i},\Psi(\boldsymbol{u}_i) ,(\boldsymbol{{u}_k},\Psi(\boldsymbol{u}_k)), \quad i,k=1, \ldots, n.
\end{equation}

The so-constructed interpolants are then evaluated on a grid of $N \times N$ data.
Precisely, let $\tilde{\cal{U}}= \{\tilde{\boldsymbol{u}}_i=(\tilde{u}_i,\tilde{v}_i), \hskip 0.2cm i=1,\ldots,N^2\} \subseteq  U$, for $i=1,\ldots,N^2$, we compute 
\begin{align}
P^{\Psi}_{_{{\rm Re}({V})}}(\tilde{\boldsymbol{u}}_i) = \boldsymbol{\kappa}\left(\tilde{\boldsymbol{u}}_i\right) \boldsymbol{\alpha^{\Psi}}, \quad {\rm and} \quad P^{\Psi}_{_{{\rm Im}({V})}}(\tilde{\boldsymbol{u}}_i) = \boldsymbol{\kappa}\left(\tilde{\boldsymbol{u}}_i\right) \boldsymbol{\beta^{\Psi}},
\label{ev_RBF}
\end{align}
where
\begin{equation*}
\boldsymbol{\kappa} \left(\tilde{\boldsymbol{u}}_i\right) = \left( K_{\varepsilon}((\tilde{\boldsymbol{u}}_i,\Psi(\tilde{\boldsymbol{u}}_i),({\boldsymbol{u}}_1,\Psi({\boldsymbol{u}}_1)), \ldots, K_{\varepsilon}((\tilde{\boldsymbol{u}}_i,\Psi(\tilde{\boldsymbol{u}}_i),({\boldsymbol{u}}_n,\Psi({\boldsymbol{u}}_n)) \right).
\end{equation*}

To fully define our interpolants, we now need to discuss the selection of both the scaling function and the kernel. Both choices are tailored to the inversion of the Fourier transform with limited data, meaning that the aim is to define a scaling function that provides an enriched interpolation space and that, at the same time, keeps the usual ill-conditioning of the kernel matrices low. 

\subsubsection{VSKs for inverse problems.}

To improve the accuracy of the interpolation procedure, we should select a scaling function that mimics the samples. Indeed, as shown in \cite{vskmpi}, this enables us to preserve shape properties of the target function. In view of this, we take a first and possibly rough approximation of $I$, namely ${\bar I }$. For stability purposes ${\bar I }$ might be first segmented as follows:
\begin{equation} \label{tres}
{\bar I }({\boldsymbol{x}}) = \left\{
    \begin{array}{ccc}
        {\bar I }({\boldsymbol{x}}), & \textrm{if} & |{\bar I }_{\epsilon}({\boldsymbol{x}})|>p \max_{\boldsymbol{x} \in X}{\bar I }({\boldsymbol{x}}),  \\
        0, & \textrm{elsewhere}, & \\
    \end{array} \right.
\end{equation}
where $p$ is a given threshold and ${\boldsymbol{x}} \in X$. Then, the scaling function will be defined by solving the forward problem \eqref{eq:00}, i.e. we compute
\begin{equation*} 
    \bar{V}(\boldsymbol{u}) = ({\cal F}\bar{I})(\boldsymbol{x}),
\end{equation*}
and we set $\Psi(\boldsymbol{u})=({\rm Re}(\bar{V}(\boldsymbol{u}) ),{\rm Im}(\bar{V}(\boldsymbol{u})))$. 

This step certainly allows us to add new features to the original scattered data. Nevertheless, since we might need to deal with noisy data, we also have to select proper kernel bases. 

\subsection{The Mat\'ern VSK}

Among several kernels which differ in terms of regularities, we here select the Mat\'ern $C^0$ RBF whose formula is given by 
$$
K_{\varepsilon}(\boldsymbol{w}, \boldsymbol{z})={\rm e}^{-\varepsilon ||\boldsymbol{w}- \boldsymbol{z} ||_2}.
$$
This selection is motivated by the fact that the Mat\'ern $C^0$ RBF is characterized by a low regularity and it is thus a reasonable choice for eventually dealing with noisy data. On the opposite, the interpolants constructed via smooth kernels might suffer from instability due to the ill-conditioning of the kernel matrices. Moreover, in the following we will show that the Mat\'ern VSK might enable us to improve expressiveness and stability of the standard setting. This theoretically motivates  the choice of the kernel. 

\subsubsection{Expressiveness of the Mat\'ern kernel.}

The concept of expressiveness is related to the complexity of a kernel-based model. Several studies show that the so-called VC dimension \cite{Vapnik71} and the empirical Rademacher complexity \cite{Bartlett02} are popular complexity indicators. We further remark that complex models are able to express sophisticated links between the data \cite{Scholkopf02}.

To better investigate the concept of expressiveness in our setting,  we introduce the so-called \emph{spectral ratio}  \cite{Donini}. For a given kernel matrix of the form \eqref{kernel}, this is defined as
$$S(\mK)=\dfrac{\|\mK\|_{\mT}}{\|\mK\|_{\mF}}=\dfrac{\sum_{i=1}^{N}\mK_{ii}}{\sqrt{\sum_{i=1}^{N}\sum_{j=1}^{N}\mK_{ij}^2}}.$$

The following definition \cite[Definition 1, p. 8]{Donini} states that the spectral ratio can be used as an expressiveness measure for kernels.

\begin{definition}
Let $K_{\varepsilon}^{(1)}$ and $K_{\varepsilon}^{(2)} :  U \times  U \longrightarrow \mathbb{R}$ be two (strictly) positive definite kernels. We say that $K_{\varepsilon}^{(2)}$ is more
specific (or more expressive) than $K_{\varepsilon}^{(1)}$ whenever for any dataset ${\cal U}=\{{\boldsymbol u}_1,\dots,{\boldsymbol u}_n\} \subseteq  U$, we have 
$$
S(\mK^{(1)}) \leq S(\mK^{(2)}).
$$
\end{definition}

We now prove that the Mat\'ern $C^0$ VSK is more expressive than the classical one. To this aim, we need to compare the spectral ratios of the kernel matrices defined in \eqref{kernel} and \eqref{kernel1}. 

\begin{proposition}\label{prop1}
Let ${\cal U}=\{ {\boldsymbol u}_i, \hskip 0.1cm i=1,\ldots,n \} \subseteq  U $ be a set of distinct data. Let $\Psi:U\longrightarrow \Sigma$ be the scaling function for the VSK setting. Let $K_{\varepsilon}:  U\times  U \longrightarrow \mathbb{R}$ be the Mat\'ern $C^0$ kernel, 
then the VSK kernel $K_{\varepsilon}^{\Psi}: U\times  U \longrightarrow \mathbb{R}$ is more expressive than $K_{\varepsilon}$.  
\end{proposition}
\begin{proof}
Being the Mat\'ern kernel non-increasing we obtain
$$
\mK_{ij} \geq \mK^{\Psi}_{ij} \geq 0, \quad i,j=1, \ldots, n,
$$
which implies 
$$
\| \mK \|_{\mF} \geq \| \mK^{\Psi} \|_{\mF}.
$$
Moreover, since for any $\boldsymbol{w} \in  U$, $K_{\varepsilon}(\boldsymbol{w},\boldsymbol{w})=K^{\Psi}_{\varepsilon}(\boldsymbol{w},\boldsymbol{w})=1$, i.e. $\mK^{\Psi}_{ii}=\mK_{ii}=1$, $i=1,\ldots, n$,  we have that  
$$
\|\mK^{\Psi}\|_{\mT}= \|\mK\|_{\mT} = n.
$$
Therefore, 
$$
S(\mK)=\dfrac{n}{\|\mK\|_{\mF}} \leq \dfrac{n}{\|\mK^{\Psi}\|_{\mF}}=S(\mK^{\Psi}).
$$
\end{proof}

On the one hand, being the Mat\'ern $C^0$ VSK more expressive than the standard one, the VSK-based model might be able to deal with more complex interpolation tasks. On the other hand, too complex models might lead to instability. However, under several hypothesis, we are able to prove that the VSK setting might improve also the conditioning of the kernel matrices, leading to a robust interpolation tool. 

\subsubsection{Spectrum of the Mat\'ern kernel.}

To investigate the condition number of the interpolation matrices generated via the Mat\'ern kernel in the VSK setting, we will make use of the following result by Schur \cite{Schur} that can be traced back to 1911 \cite[Lemma A.5]{Karoui}.

\begin{theorem} \label{th_eig}
If $\mE$ and $\mM \in \mathbb{R}^{n\times n}$  are positive definite matrices, denoting by $\lambda_{\min}$ and $\lambda_{\max}$ the smallest and largest eigenvalue of a matrix, we have that
$$\lambda_{\min}(\mE) \min_{i=1, \ldots, n} \mM_{ii} \leq \lambda_i(\mE \circ \mM) \leq \lambda_{\max}(\mE) \max_{i=1, \ldots, n} \mM_{ii},$$
where $\circ$ is the entry-wise product of matrices.
\end{theorem}
To infer on the spectrum of the Mat\'ern $C^0$ kernel, we further have to introduce the Gaussian $C^{\infty}$ RBF. We remark that the Gaussian kernel $K_{\varepsilon,G}:  U\times  U \longrightarrow \mathbb{R}$  is defined as $K_{\varepsilon,G}= {\rm e}^{-\varepsilon^2 ||\boldsymbol{w}-\boldsymbol{z} ||^2_2}$. For such kernel, we observe that its VSK matrix is given by 
\begin{equation}\label{gauss}
\mK_G^{\Psi}= \mK_G \circ \mK_G^{\varphi},     
\end{equation}
where $\mK^{\varphi}_{ij}={\rm e}^{-\varepsilon \lVert\Psi(\boldsymbol{u}_i)-\Psi(\boldsymbol{u}_j)\lVert_2^2}$, $i,j=1, \ldots, n$, and $\mK_G$ is the Gaussian kernel matrix based upon the distance matrix $\mD$ given by
$$
\mD_{ij} = \lVert \boldsymbol{u}_i-\boldsymbol{u}_j \lVert_2^2, \quad i,j=1, \ldots, n.
$$
Similarly, the distance matrix in the VSK setting is defined as 
$$
\mD^{\Psi}_{ij} = \lVert (\boldsymbol{u}_i,\Psi(\boldsymbol{u}_i))-(\boldsymbol{u}_j,\Psi(\boldsymbol{u}_j)) \lVert_2^2, \quad i,j=1, \ldots, n.
$$

We now have all the ingredients to study the condition number of the VSK matrix generated via the Mat\'ern $C^0$ kernel. For a given matrix $\mM$, we focus on the $2$-condition number defined as
\begin{equation*}
    \textrm{cond}(\mM)= ||\mM||_2||\mM^{-1}||_2.
\end{equation*}

\begin{proposition} \label{prop2}
Let ${\cal U}=\{ {\boldsymbol u}_i, \hskip 0.1cm i=1,\ldots,n \} \subseteq  U $ be a set of distinct data. Let $\Psi:U \longrightarrow \Sigma$ be the scaling function for the VSK setting. Let $K_{\varepsilon}:  U\times  U \longrightarrow \mathbb{R}$ be the Mat\'ern $C^0$  kernel. Given the VSK matrix $\mK^{\Psi}$  constructed via $K_{\varepsilon}^{\Psi}: { U} \times  U \longrightarrow \mathbb{R}$, if $$\mK^{\circ (\varepsilon \mD-1)} \circ \mK_G^{\varphi} \circ \left( {(\mK^{\Psi})}^{\circ (\varepsilon \mD^{\Psi}-1)} \right)^{\circ -1},$$ is positive definite, we have that 
$$
{\rm cond}(\mK^{\Psi}) \leq {\rm cond} (\mK).
$$
\end{proposition}
\begin{proof}
At first we need to point out the link between the Gaussian and Mat\'ern kernel matrices. With the notation previously introduced, denoted by $\mK$ the Mat\'ern $C^0$ kernel matrix, we have that
\begin{equation}\label{mat1}
\mK_G = \mK \circ (\mK^{\circ (\varepsilon \mD-1)}), 
\end{equation}
and analogously, 
\begin{equation}\label{mat2}
\mK_G^{\Psi}= \mK^{\Psi} \circ (\mK^{\Psi})^{\circ (\varepsilon \mD^{\Psi}-1)}.
\end{equation}
Then, from \eqref{gauss} and \eqref{mat1} we obtain
$$
\mK_G^{\Psi}= \mK_G \circ \mK_G^{\varphi} = \mK \circ (\mK^{\circ (\mD-1)}) \circ \mK_G^{\varphi} ,
$$
which, thanks to \eqref{mat2}, implies that
$$
\mK^{\Psi} \circ (\mK^\Psi)^{\circ (\varepsilon \mD^{\Psi}-1)} = \mK \circ (\mK^{\circ  (\varepsilon \mD-1)}) \circ \mK_G^{\varphi},
$$
and thus,
$$
\mK^{\Psi}  = \mK \circ \left(\mK^{\circ (\varepsilon \mD-1)} \circ \mK_G^{\varphi} \circ \left( (\mK^{\Psi})^{\circ (\varepsilon \mD^{\Psi}-1)} \right)^{\circ -1}\right).
$$
Furthermore, since the Mat\'ern kernel is strictly positive definite the associated kernel matrices are positive definite and thus the condition number can be computed as
$$
{\rm cond}(\mK^{\Psi})  = \dfrac{{\lambda}_{\max}(\mK^{\Psi})}{{\lambda}_{\min}(\mK^{\Psi})},
$$
and since  
$$\left(\mK^{\circ (\varepsilon \mD-1)} \circ \mK_G^{\varphi} \circ \left( {(\mK^{\Psi})}^{\circ (\varepsilon \mD^{\Psi}-1)} \right)^{\circ -1}\right)_{ii}=1, \quad i = 1,\ldots,n,$$
from Theorem \ref{th_eig}, we obtain
$$
{\rm cond}(\mK^{\Psi}) = \dfrac{{\lambda}_{\max}(\mK^{\Psi})}{{\lambda}_{\min}(\mK^{\Psi})} \leq \dfrac{{\lambda}_{\max}(\mK)}{{\lambda}_{\min}(\mK^{\Psi})}  \leq \dfrac{{\lambda}_{\max}(\mK)}{{\lambda}_{\min}(\mK)}= {\rm cond} (\mK).
$$
\end{proof}

These results, together with the error bounds shown in \cite{vskmpi}, give a theoretical validation for the use of VSKs. Numerical tests, carried out in the next section, show that the proposed method can be effectively used for many inverse problems, provided that the scaling function is appropriately selected for the considered application. As an example, in the following we test the method in the framework astronomical imaging.

\section{Applications to astronomical imaging}
\label{Application}

Fourier-based imaging finds its main applications in medical and astronomical imaging \cite{lustig2007sparse,Duijndam,richard2017interferometry}. For the latter topic, the image reconstruction problem can be designed as follows.

Let $\boldsymbol{u} \in U$, $U \subseteq \mathbb{R}^2$, be a point in the spatial frequency plane, and $I(\boldsymbol{x})$ the source function corresponding to the point $\boldsymbol{x}=(x, y)$ belonging to the physical plane $X \subseteq \mathbb{R}^2$. Given the following relation, 
\begin{equation}\label{eq:001}
    V (\boldsymbol{u}) = \int_{\mathbb{R}^2} I(\boldsymbol{x}) {\rm e}^{2 \pi i \boldsymbol{u} \cdot  \boldsymbol{x}} {\rm d} \boldsymbol{x},
\end{equation}
and some scattered observations on a compact set $D \subseteq U$ of the function $V$, the image reconstruction problem consists in finding an approximation of $I$. 

In the present study we focused on solar hard X-ray imaging and, specifically, we cast our interpolation-based reconstruction scheme into the framework of the NASA {\em{RHESSI}} and the ESA {\em{STIX}} missions. {\em{RHESSI}} had its nominal phase between February 2002 and August 2018 and many inversion methods have been formulated to express its observations as images \cite{Aschwanden,Benvenuto,Bonettini,Cornwell,Massa_2020,Massone2009HARDXI,felix2017compressed}. Instead, {\em{STIX}} will begin its nominal phase in September 2021 and the few studies devoted to its imaging process involve just synthetic data \cite{Massa_2020}. 

{\em{RHESSI}} and {\em{STIX}} share the same imaging concept \cite{2002SoPh,STIX1}, in which the measured counts are arranged into a set of $n$ samples of the Fourier transform of the incoming photon flux, named visibilities, each one associated to a specific point $(u,v)$ of the Fourier $(u,v)$-plane. In the case of {\em{RHESSI}}, nine collimators provided these visibilities on $9$ circles of the $(u,v)$-plane with increasing radii from about $2.73 \times 10^{-3}$  arcsec$^{-1}$ to $2.21\times 10^{-1}$  arcsec$^{-1}$, and $n$ depending on the count statistics. We point out that in the application considered in Section 5 below, we utilized detectors from 3 through 9. The radius of detector 3 is 
$7.36\times 10^{-2}$ arcsec$^{-1}$. In the case of {\em{STIX}}, $30$ subcollimators relying on the Moire pattern technology, will provide $60$ visibilities on 10 circles of the $(u,v)$-plane with increasing radii from about $2.79\times 10^{-3}$ arcsec$^{-1}$ to $7.02\times 10^{-2}$ arcsec$^{-1}$. 

In order to apply our interpolation scheme to these data, we first need to compute the scaling function $\Psi$ and the optimal shape parameter for the two instruments.

\subsection{VSKs for RHESSI and STIX}

In order to define the scaling function $\Psi$ we will make use of a first approximation of the inverse problem obtained via a standard back-projection algorithm \cite{Mersereau} that computes the discretized inverse Fourier transform of the visibilities by means of the IDL source code  {\texttt {vis\char`_bpmap}}  available in the NASA {\em{Solar SoftWare (SSW)}} tree.  Given the visibilities, the latter algorithm returns an $M \times M$ grid, namely ${\bar I }_{\epsilon}(\bar{\boldsymbol{x}}_i)$, $i=1,\ldots, M^2$. For stability purposes such an image is first segmented as in \eqref{tres}, where we numerically observed that an appropriate choice for the threshold is $0.70 \leq p \leq 0.90$. Then, by solving the forward problem \eqref{eq:001}, we obtain $ \bar{V}(\bar{\boldsymbol{u}}_i)$, $\bar{\boldsymbol{u}}_i \in U$, $i= 1, \ldots, M^2$.

To fully understand our strategy for defining the function $\Psi$ for {\em{RHESSI}} and {\em{STIX}}, we need to point out that the interpolation/extrapolation procedure provides an image of size $M \times M$, with $M=128$. Precisely, after evaluating the VSK interpolants on a grid of $N \times N$ data, we implement a zero-padding strategy that provides a grid of $T \times T$ pixels, with $T \gg N$. Finally, after the inversion we subsample the $T \times T$  grid for obtaining an $M \times M$ image. Therefore, given $ \bar{V}({\bar u}_i,\bar{v}_j)$, $i,j= 1, \ldots, M$, for maintaining the proportion between the grids we take $ \bar{V}({\bar u}_i,\bar{v}_j)$, $i,j= M/2-\lfloor L/2 \rfloor -1, \ldots, M/2+\lfloor L/2 \rfloor +1$, where $L = \lfloor T/(MN) \rfloor$. Then, those values are interpolated and evaluated at the sets ${\cal U}$ and $\tilde{{\cal U}}$. This step generates 
\begin{equation}
\label{int}
\Psi(\boldsymbol{u}_i):=(P_{{\rm Re}({\bar{V}})} (\boldsymbol{u}_i),P_{{\rm Im}(\bar{V})} (\boldsymbol{u}_i)), \quad i=1,\ldots,n,
\end{equation}
and 
\begin{equation} \label{eval}
\Psi(\tilde{\boldsymbol{u}}_i):=(P_{{\rm Re}(\bar{V})} (\tilde{\boldsymbol{u}}_i),P_{{\rm Im}(\bar{V})} (\tilde{\boldsymbol{u}}_i)), \quad i=1,\ldots,N^2.    
\end{equation}
Therefore, we are able to encode the back-projection map into the kernel and implement the VSK setting. Indeed, thanks to \eqref{int} and \eqref{eval}, we can construct the VSK kernel matrix $\mK_{\Psi}$ and evaluate the VSK interpolant as in \eqref{ev_RBF}.

\begin{remark}
For STIX visibilities, as well as for RHESSI  visibilities computed via detectors 3-9, we fix $N=320$, producing visibility grids of mesh size $5 \times 10^{-4}$ arcsec$^{-1}$. The value of $T$ is then set as $1920$. We then subsample the $T \times T$ image by covering it with $M^2$ masks of size $15\times 15$ pixels and we select the first pixel for each mask. This step leads to approximated images of pixel size of about $1$ arcsec. 
\end{remark}

We conclude this section with a comment on the selection of the shape parameter for the considered imaging problem. 
 
\subsection{Optimal shape parameter for RHESSI and STIX}

In this subsection, we propose a criterion to select an optimal shape parameter for classical kernel-based interpolation. In order to make a fair comparison between classical and VSK interpolation, the same shape parameter will be used for the VSK setting too. 

Note that Theorem \ref{powfuntheor} bounds the pointwise error in terms of the power
function which depends on the kernel and on the data points but is independent of the function values. This suggests a criterion to select a reliable shape parameter, i.e. the shape parameter that minimizes the power function computed over the data. While for {\em{RHESSI}} we should compute the power function for each data configuration (the data locations may vary during the acquisition process), for {\em{STIX}} we can provide an a priori optimal shape parameter. To this aim, we take $100$ values of the shape parameter in $[0.01,1]$ and we evaluate for each of them the corresponding maximum value of the power function. With the considered Mat\'ern $C^0$ kernel, the result is the one plotted in Figure \ref{fig1_new}. We note that, having a few data by {\em{STIX}} makes the problem quite stable; indeed the error curve grows monotonically with respect to the shape parameter. Therefore, we fix $\varepsilon=0.01$. 

\begin{figure}
		\centering 
		 \includegraphics[scale= 0.28]{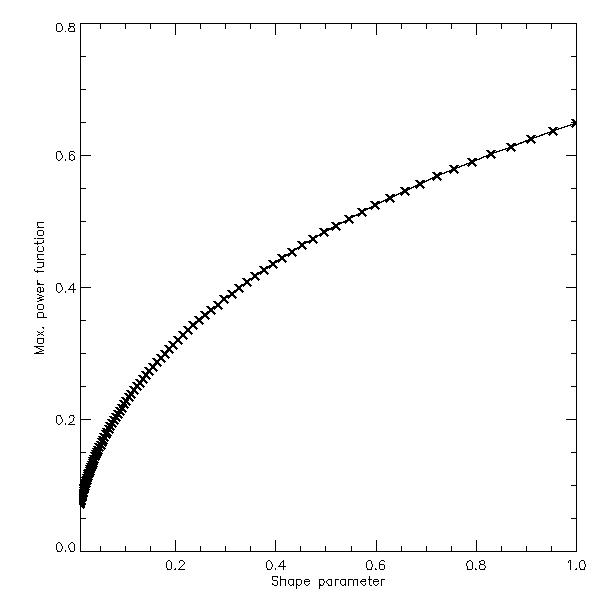} 
		\caption{The shape parameter VS the maximum value of the power function computed for the {\em{STIX}} visibilities.} 
		\label{fig1_new} 
\end{figure}

Despite the fact that the optimal shape parameter is investigated only for {\em{STIX}}, numerical evidence shows that such value is reliable for {\em{RHESSI}} too. 

\section{Numerical experiments}
\label{experiments}
We first consider a motivating example devoted to stress the dependence of the proposed algorithm on the interpolation routine and more specifically on the locations of the sampling nodes. We then consider an example involving {\em{STIX}} synthetic visibilities simulated from a realistic flaring source model and an application on experimental {\em{RHESSI}} data. 

For all the experiments, we test and compare two reconstruction algorithms both based on the Landweber scheme with two different interpolation procedures: the first, namely Land-RBF, uses classical RBFs and the second, Land-VSK, implements the VSK strategy.

\subsection{Motivating example}
Let us consider the two dimensional Gaussian function
\begin{equation}\label{gaussian}
I({\boldsymbol{x}}) = A {\rm e}^{-B \|{\boldsymbol{x}}-{\boldsymbol{x}}_p\|}~,
\end{equation}
with ${\boldsymbol{x}}_p$, $B$ and $A$ chosen in such a way to mimic an X-ray emitting source from position ${\boldsymbol{x}}_p$ on the Sun, Full Width at Half Maximum (FWHM) equal to 11 arcsec and photon flux equal to $10^4$ photons cm$^{-2}$ s$^{-1}$. Using the Monte Carlo code at disposal in the {\em{STIX}} simulation software we produce three sets of visibilities according to the $(u,v)$ samplings represented in Figure \ref{fig0_b}. The left panel of this figure corresponds to $100$ Fibonacci nodes, i.e. well-spaced nodes in the $(u,v)$-plane. We remark that the problem of finding optimal data locations for kernel-based interpolation problems is a well-known issue studied by many researchers (see, e.g., \cite{Wendland05}). Without giving details, the Fibonacci nodes have the property of being well-distributed, and hence any RBF-based interpolation routine should work properly on them. The middle panel corresponds to the {\em{RHESSI}} sample and, in this experiment, we assumed that $n=240$ visibilities are provided by {\em{RHESSI}}. Finally, the right panel corresponds to the sampling of $n=60$ visibilities performed by {\em{STIX}}. Correspondingly to these three sampling configurations, back-projection provides the three images represented in Figure \ref{Fig_fib_rhessi_stix_BPMAPS}. Such maps are used as starting point for constructing the VSK scaling functions. Then the outcomes for VSK and classical interpolations are plotted in Figure \ref{Fig_fib_rhessi_stix_vsk} and Figure \ref{Fig_fib_rhessi_stix}, respectively, where we also plotted the interpolated modulus of the visibility surfaces with the different methods. 

\begin{figure}
\centering
    \includegraphics[scale=0.20]{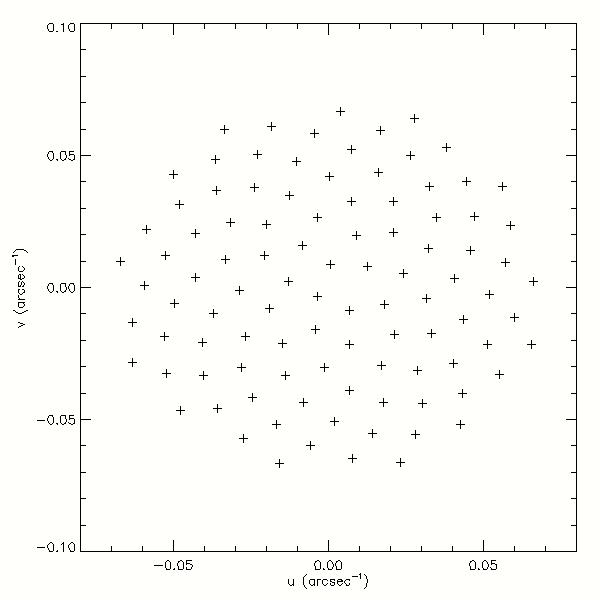} 
    \includegraphics[scale=0.20]{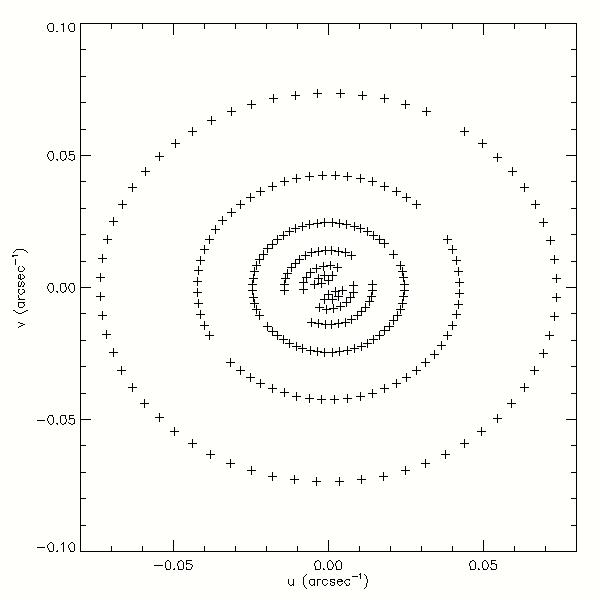}
    \includegraphics[scale=0.20]{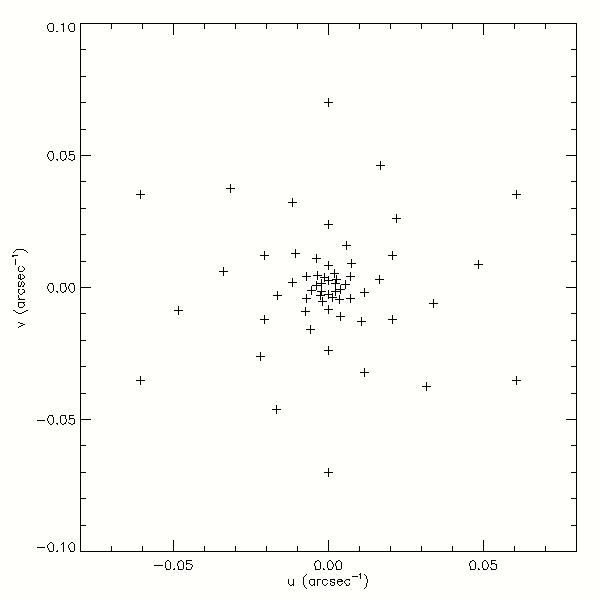} 
    \caption{Illustrative example of Fibonacci (left), {\em{RHESSI}} (middle) and {\em{STIX}} (right) visibilities.}
    \label{fig0_b}
\end{figure}

We observe that as long as the data are well-spaced (as in the case of Fibonacci nodes), both reconstructions return comparable results; however, note that the highest peak intensity is not correctly detected by Land-RBF. The differences between the two schemes become more and more evident when interpolating {\em{RHESSI}} and {\em{STIX}} visibilities.

\begin{figure}
    \centering
  \includegraphics[scale=0.14]{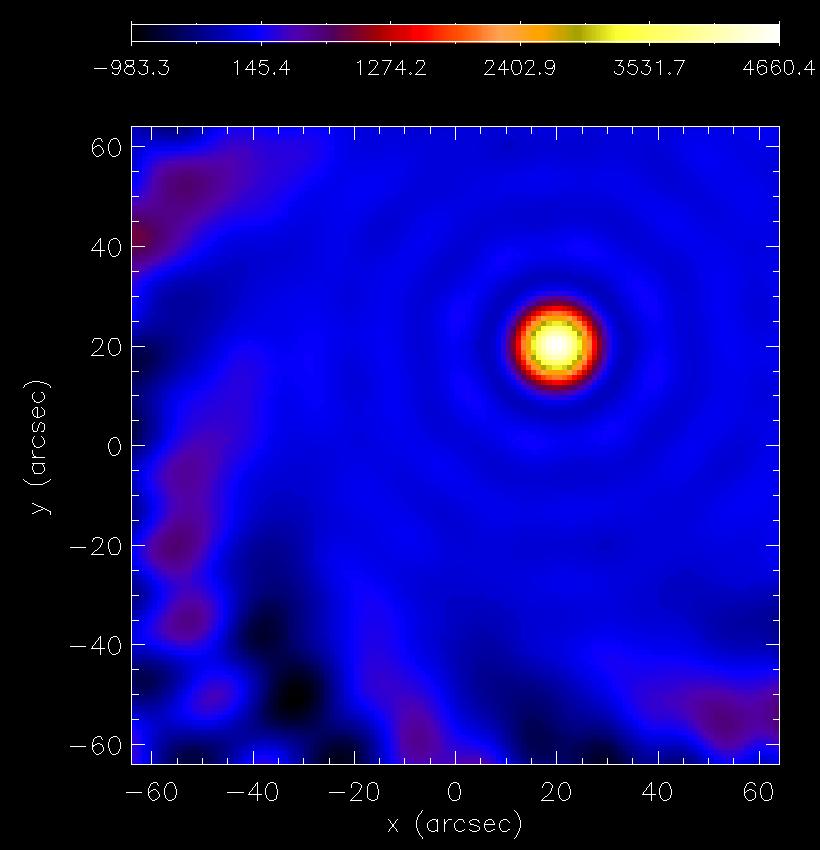}    
  \includegraphics[scale=0.14]{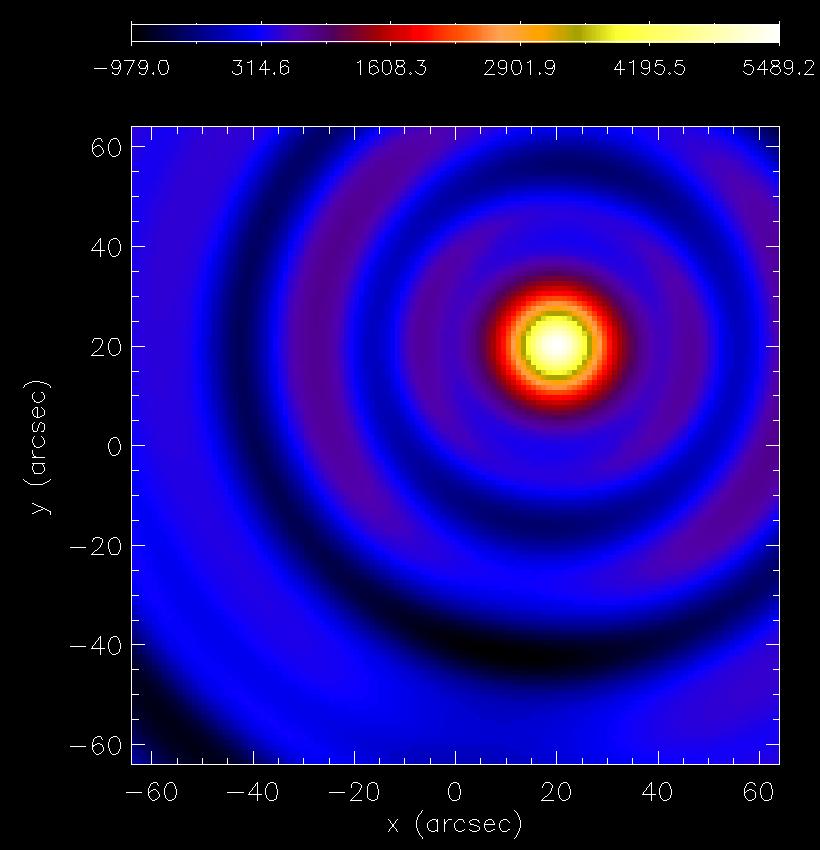} 
 \includegraphics[scale=0.14]{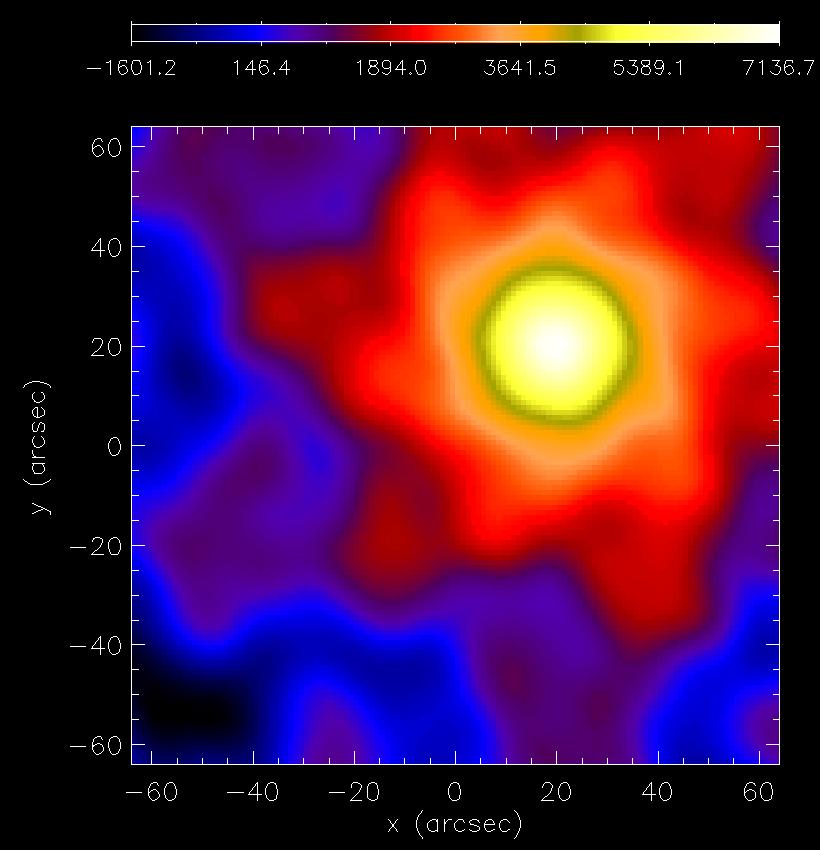}  
\caption{Reconstruction of the Gaussian source function (\ref{gaussian}). From left to right: back-projection map computed via Fibonacci, {\em{RHESSI}} and {\em{STIX}} data locations, respectively. These back-projection maps have been used to generate the scaling functions exploited by Land-VSK.}
\label{Fig_fib_rhessi_stix_BPMAPS} 
\end{figure}

\begin{figure}
    \centering
  \includegraphics[scale=0.13]{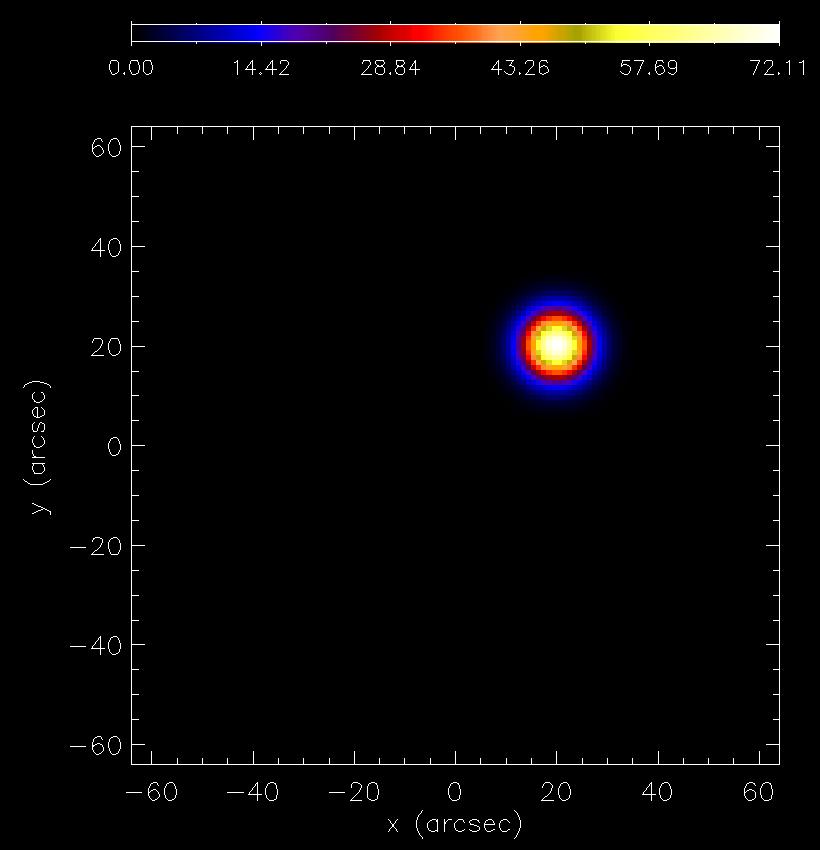}        
    \includegraphics[scale=0.13]{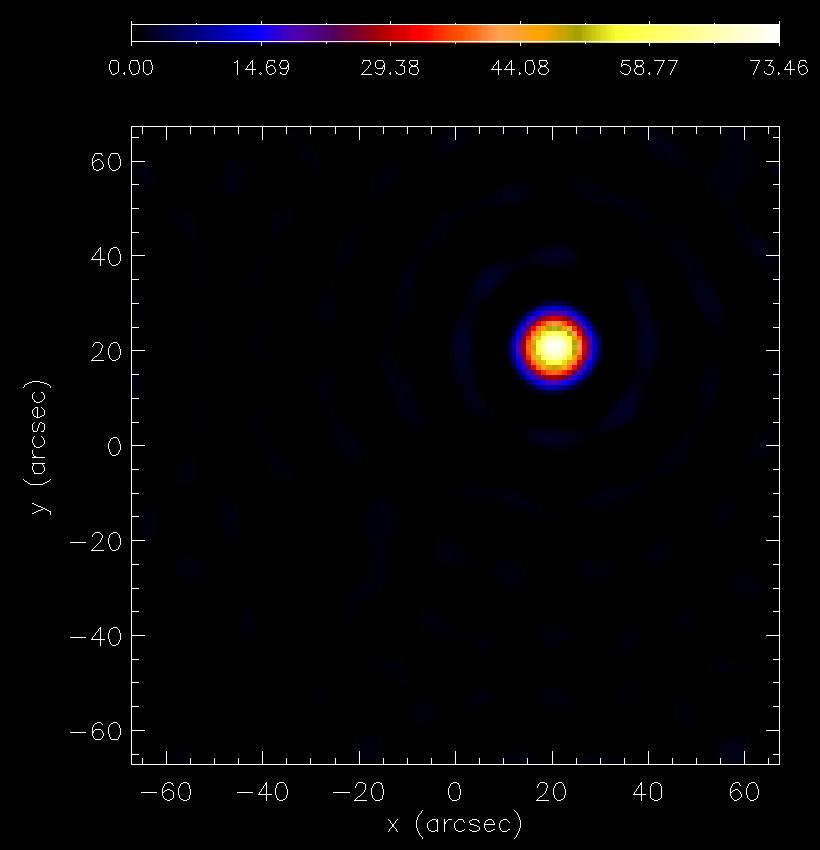} 
    \includegraphics[scale=0.13]{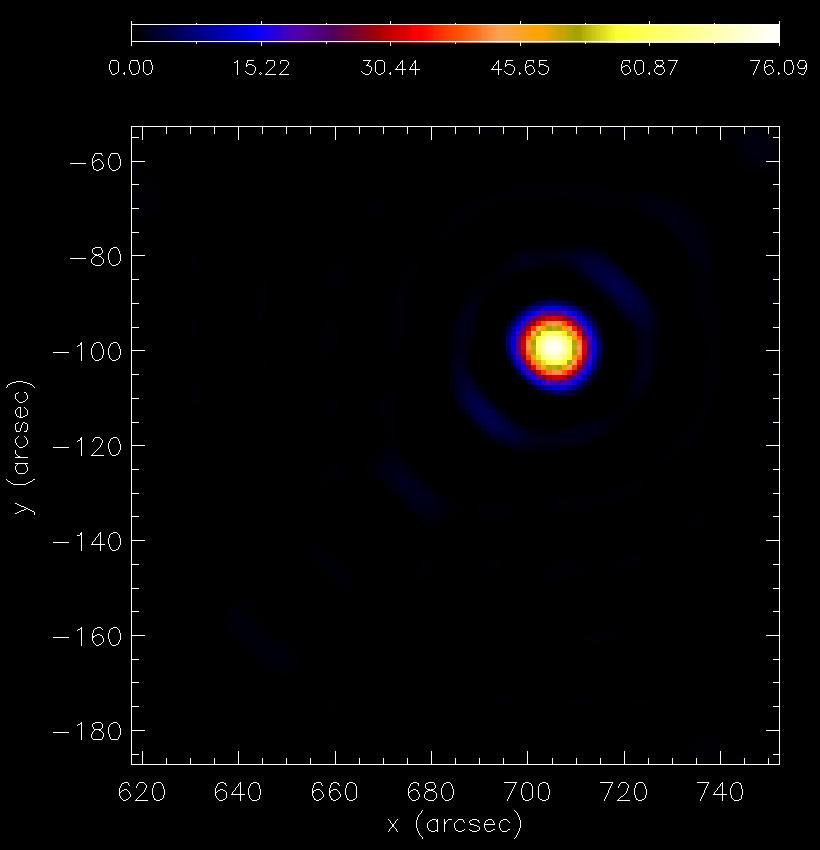}    
   \includegraphics[scale=0.13]{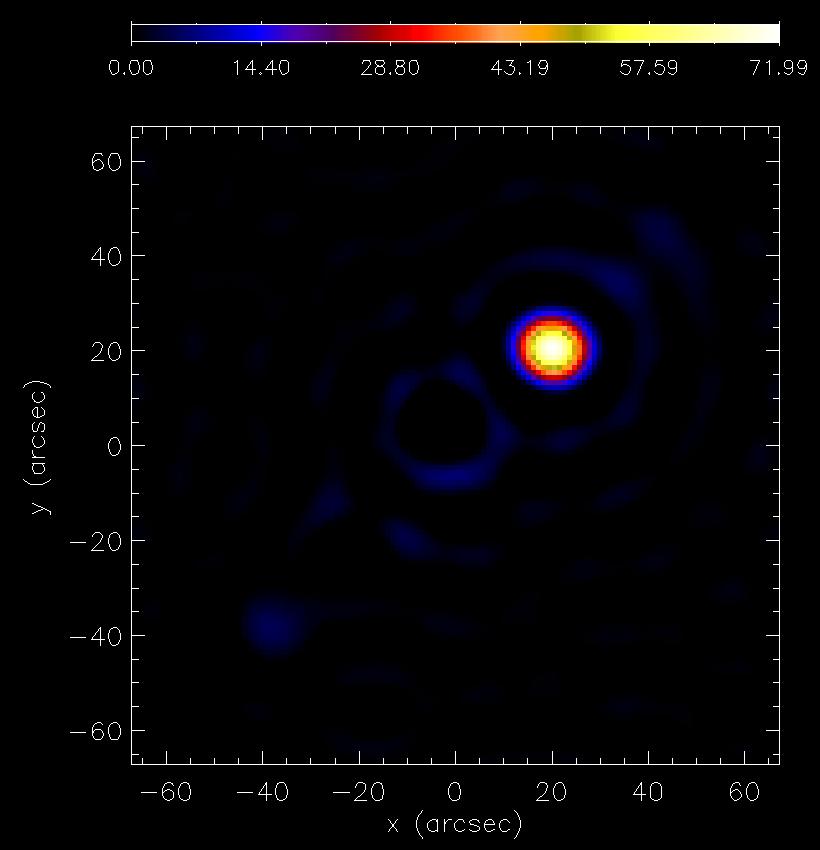} \\
    \includegraphics[scale=0.13]{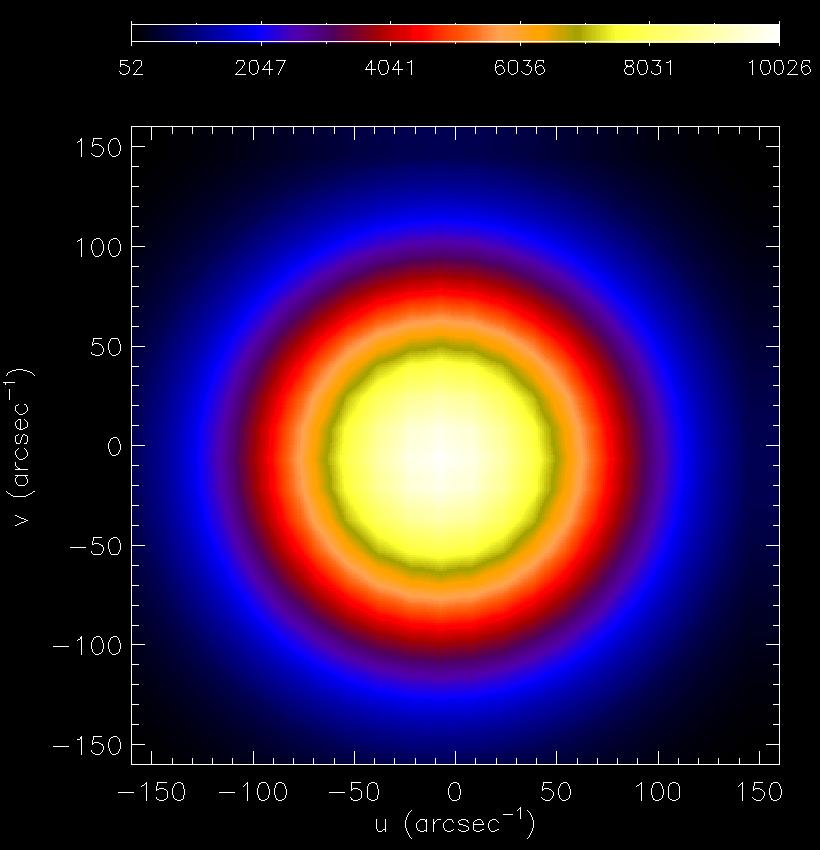} 
  \includegraphics[scale=0.13]{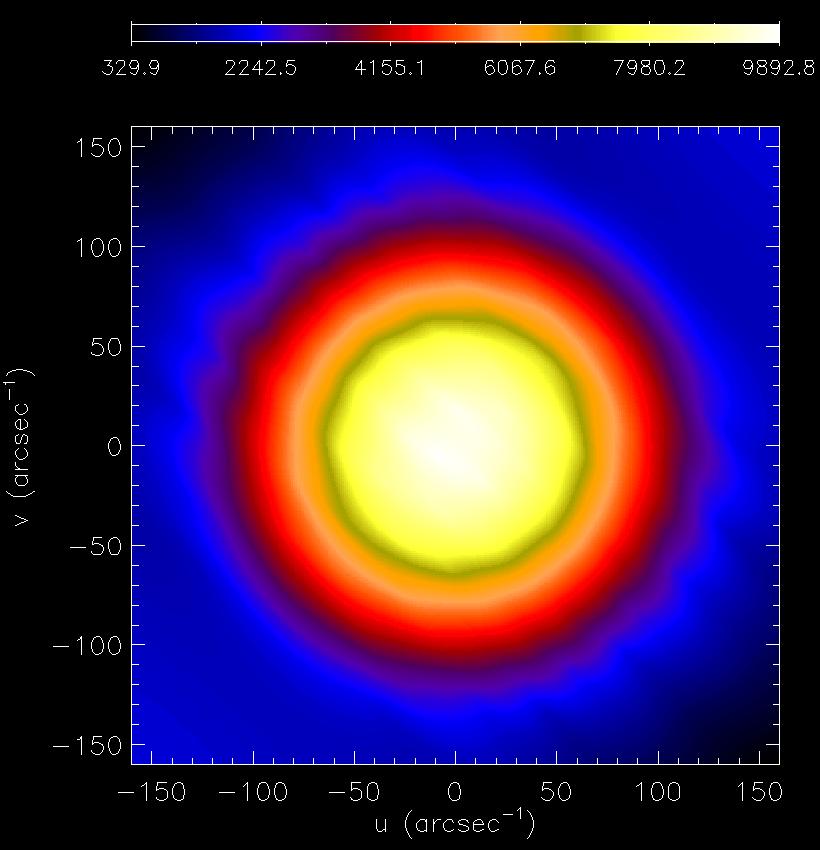}   
   \includegraphics[scale=0.13]{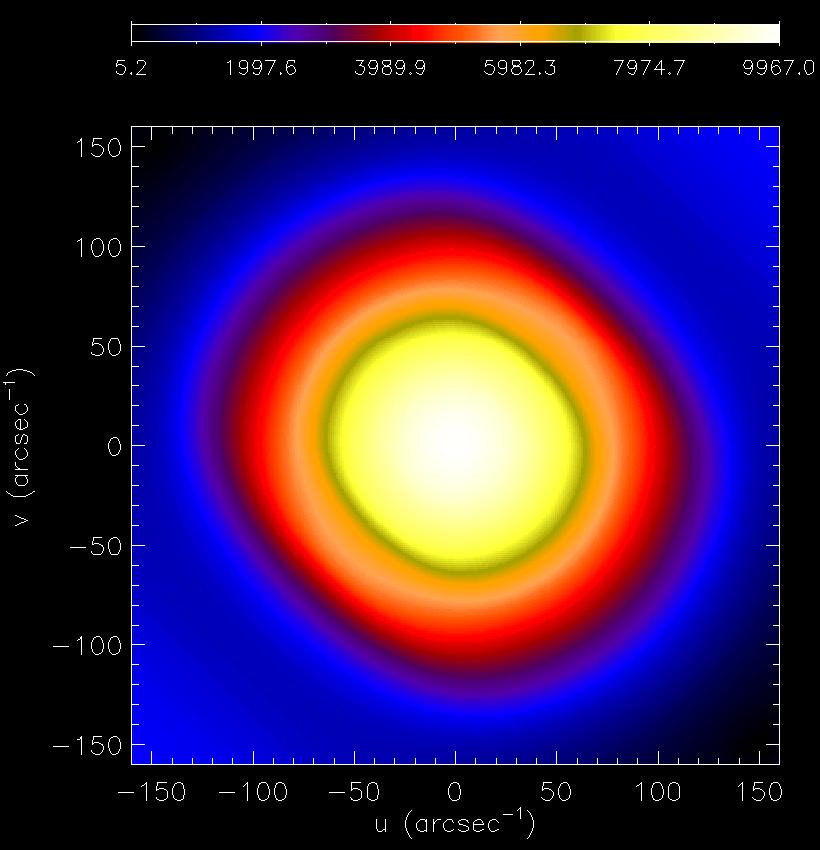}   
  \includegraphics[scale=0.13]{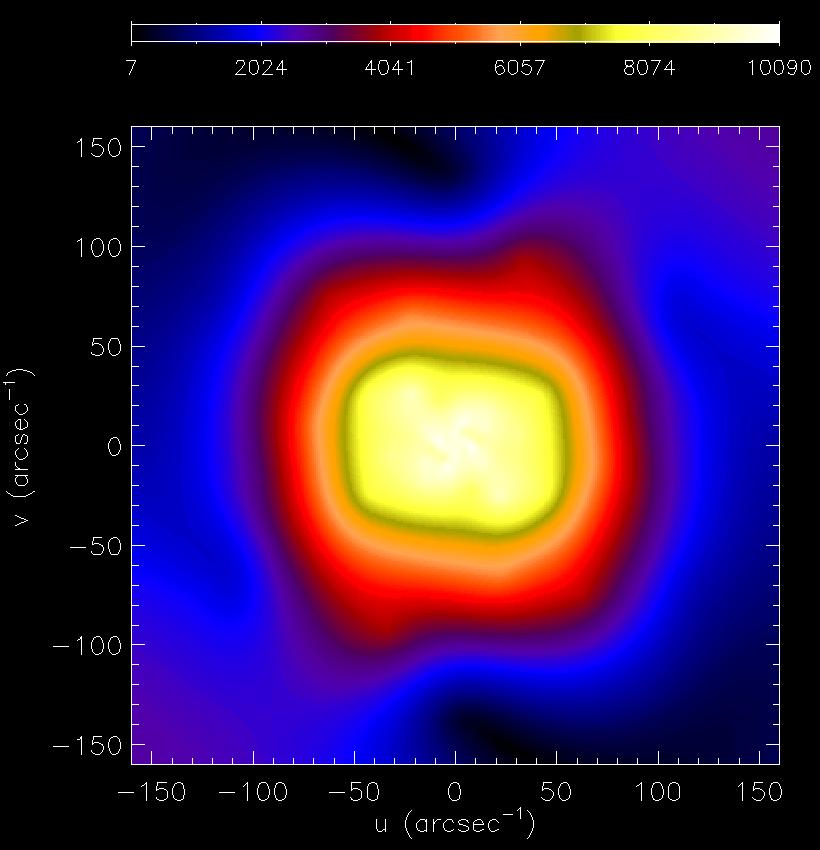} 
  \caption{Top row, from left to right: ground truth and reconstructions of the Gaussian source function via Land-VSK using Fibocacci, {\em{RHESSI}} and {\em{STIX}} data locations. The second row represents the visibility surfaces of the ground truth and the ones returned by the interpolation algorithm.}
    \label{Fig_fib_rhessi_stix_vsk}
\end{figure}

\begin{figure}
    \centering
  \includegraphics[scale=0.13]{./FigPaper/2pointL_1_realmap}    
  \includegraphics[scale=0.13]{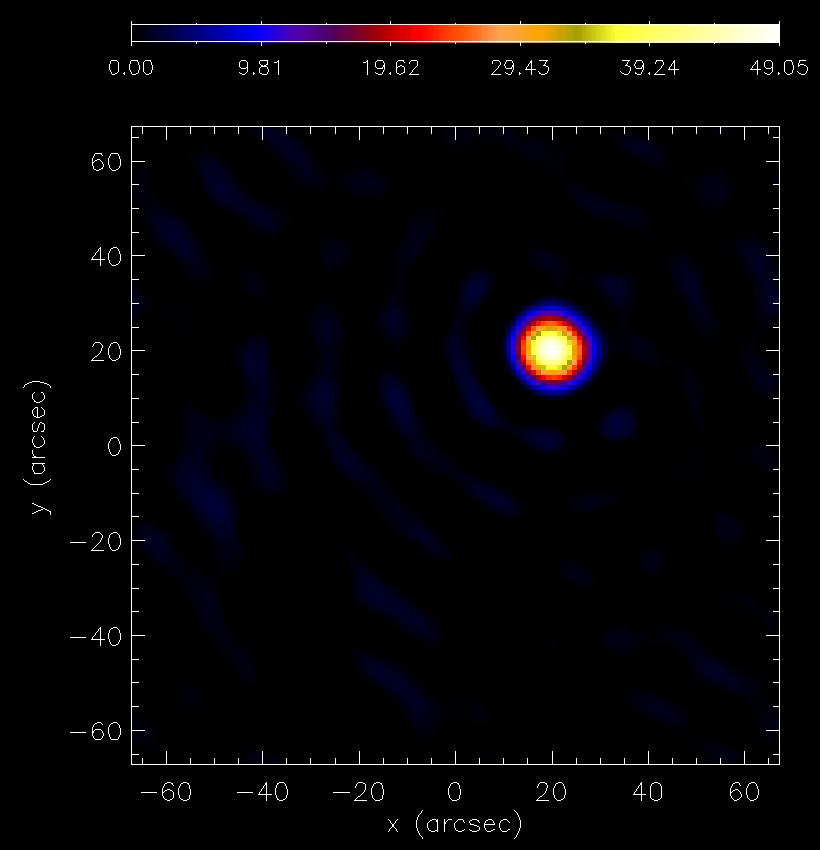} 
      \includegraphics[scale=0.13]{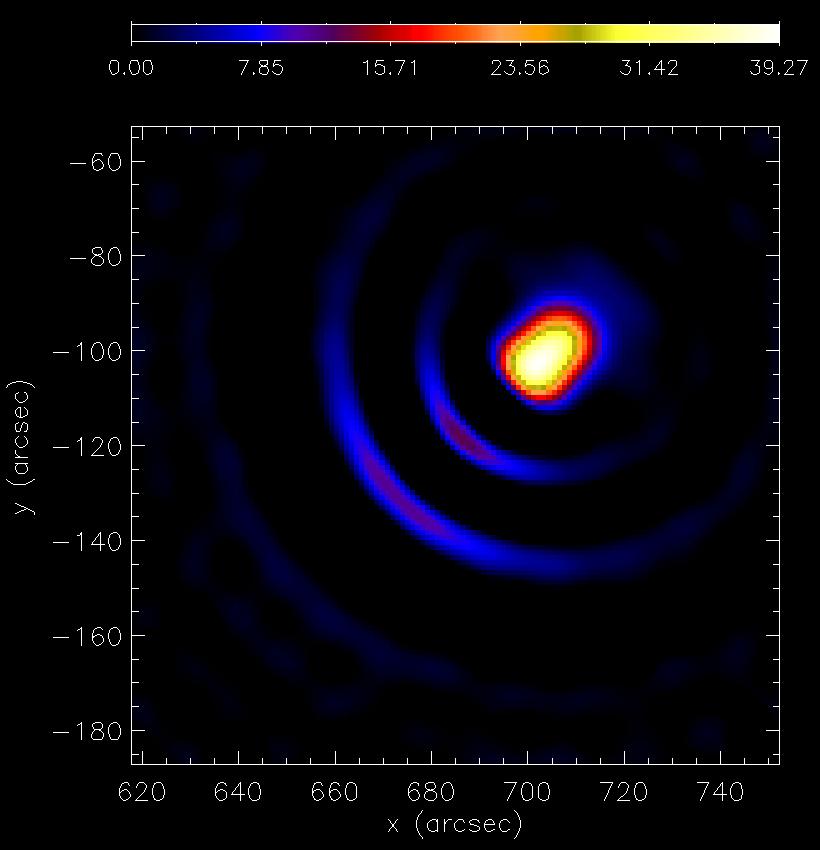} 
     \includegraphics[scale=0.13]{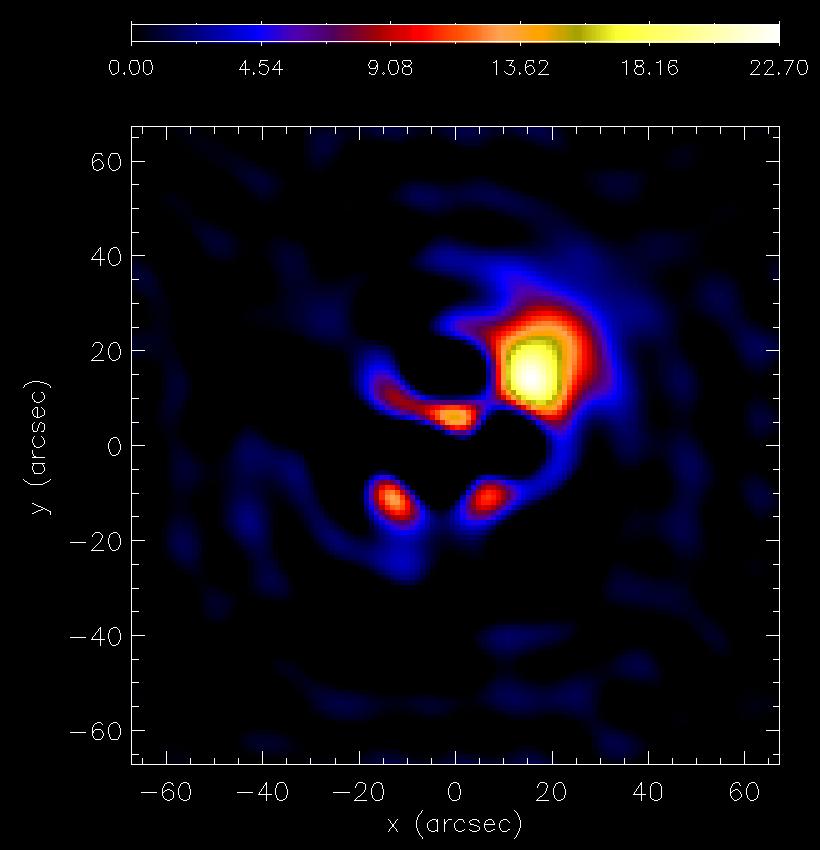}\\
  \includegraphics[scale=0.13]{./FigPaper/2pointL_1_realmap_vism}     
  \includegraphics[scale=0.13]{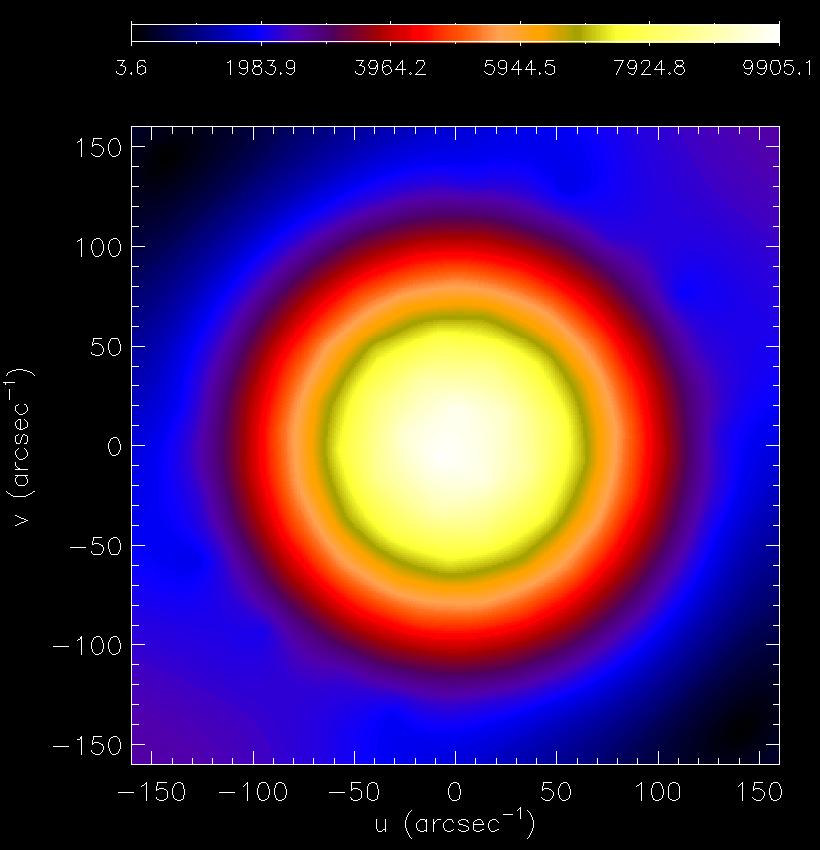}  
  \includegraphics[scale=0.13]{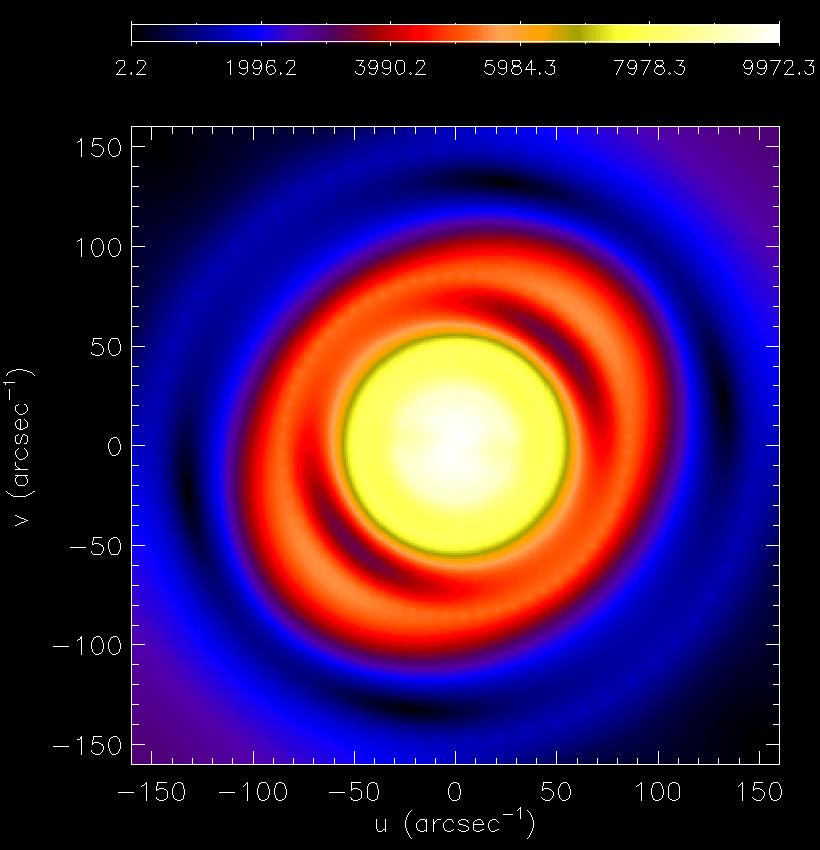}    
  \includegraphics[scale=0.13]{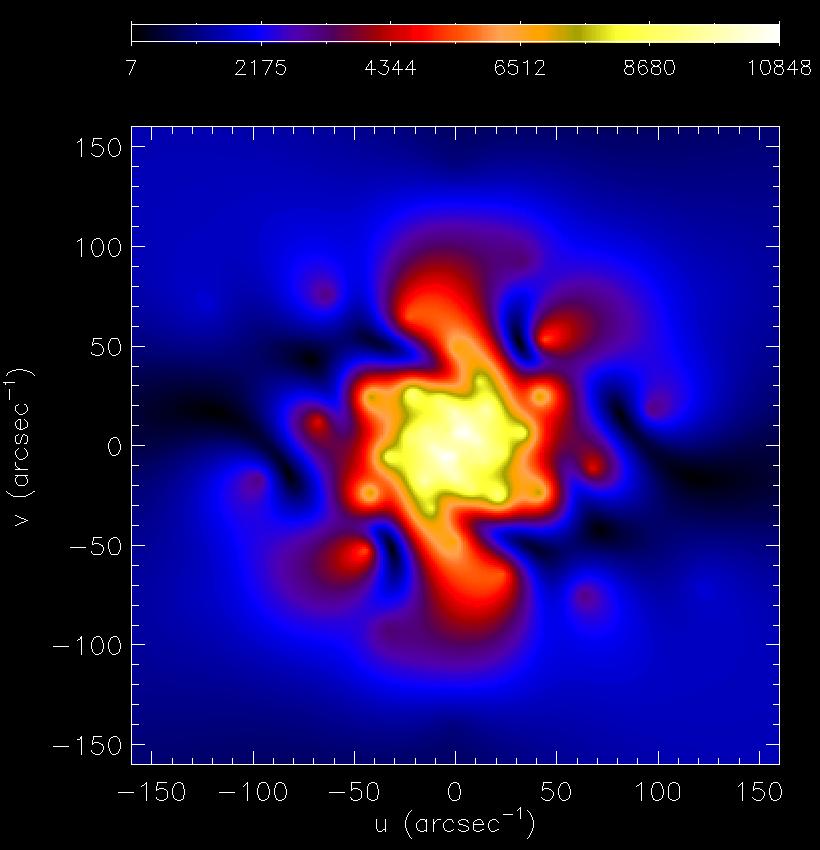}  
  \caption{Top row, from left to right: ground truth and reconstructions of the Gaussian source function via classical Land-RBF methods using Fibocacci, {\em{RHESSI}} and {\em{STIX}} data locations. The second row represents the visibility surfaces of the ground truth and the ones returned by the interpolation algorithm.}
    \label{Fig_fib_rhessi_stix}
\end{figure}

\subsection{STIX simulated observation of a double foot-point}

Flaring emission at hard X-ray energies is typically due to the bremsstrahlung interaction between electron energies accelerated along the two arms of a magnetic loop and the ambient plasma \cite{brown1971deduction}. As a consequence, a typical flare configuration is the one of Figure \ref{fig4_new}, top left panel, where the configuration parameters are described in Table \ref{tab1_new} (ground truth). We generated $25$ sets of synthetic visibilities by as many runs of the Monte Carlo code and applied Land-RBF and Land-VSK to these synthetic sets. Figure \ref{fig4_new} compares the reconstructions and the corresponding visibility surfaces for the ground truth and the two reconstruction algorithms, while the corresponding averaged parameter estimates are illustrated in Table \ref{tab1_new} together with their standard deviations. The table also shows the Relative Root Mean Square Error (RRMSE) defined as
\begin{equation*}
{\rm RRMSE} = \dfrac{||\mV-\mW ||_{\mF}}{||\mW||_{\mF}}~,
\end{equation*}
where $V$ and $W$ are the moduli of the visibility surfaces corresponding to the ground truth and the reconstruction, respectively, and $\| \cdot \|_{\mF}$ is the Frobenius norm. Further, to numerically verify Propositions \ref{prop1} and \ref{prop2}, we report in Table \ref{tabcond_new} the spectral ratios and condition numbers for the standard and VSK Mat\'ern kernels for one out of the $25$ simulations.

\begin{figure}
    \centering
    \vskip 0.1cm 
  \includegraphics[scale=0.15]{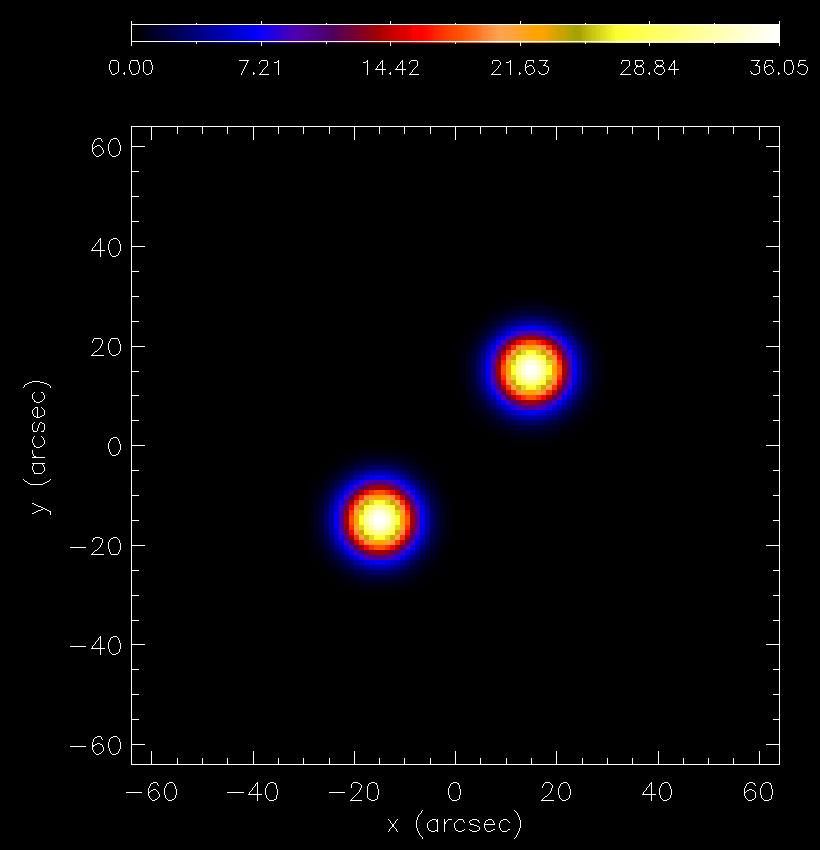}
  \includegraphics[scale=0.15]{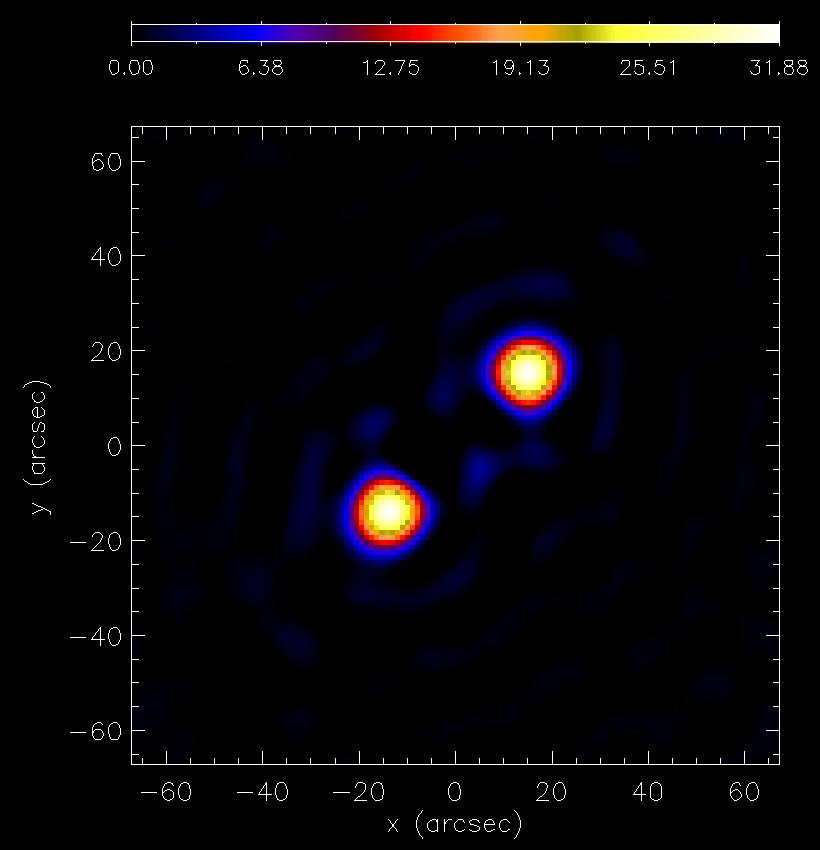}
  \includegraphics[scale=0.15]{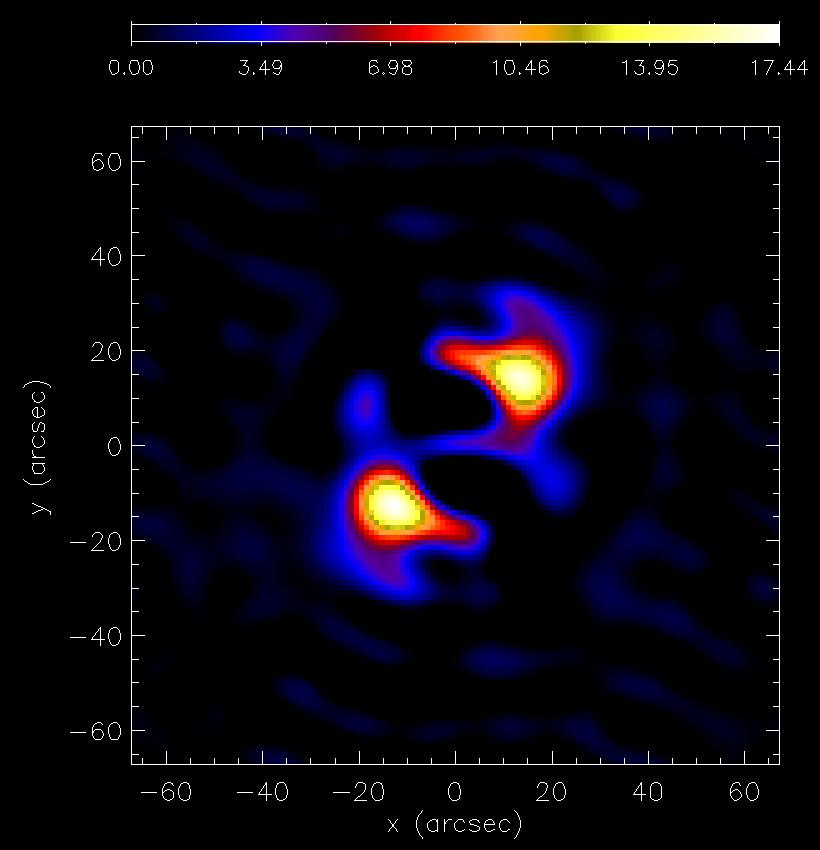}  \\
  \includegraphics[scale=0.15]{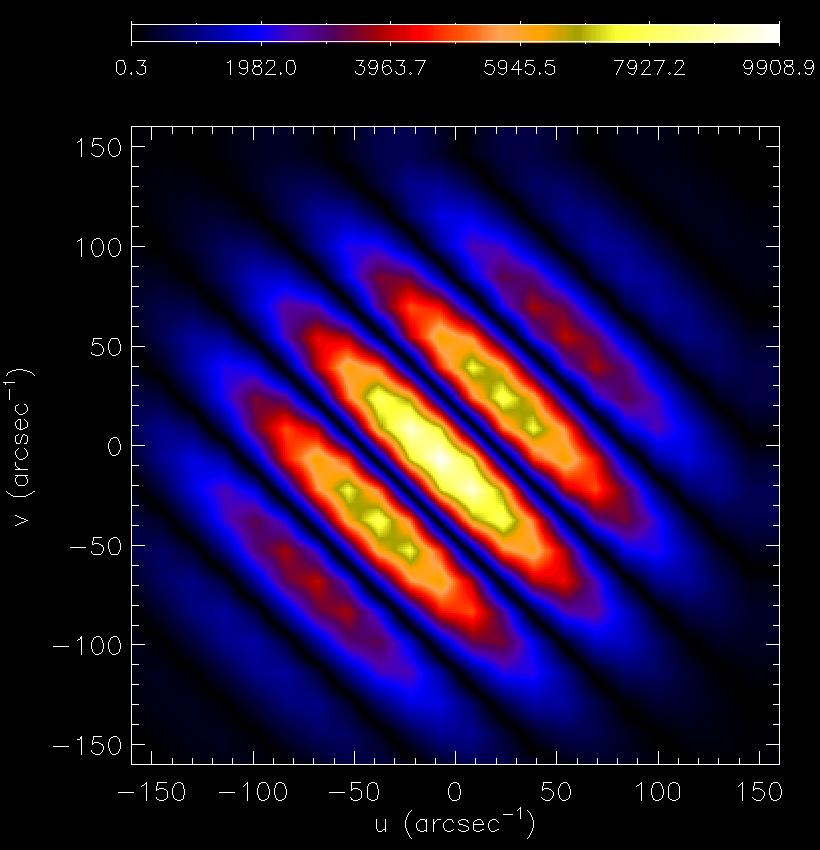}
  \includegraphics[scale=0.15]{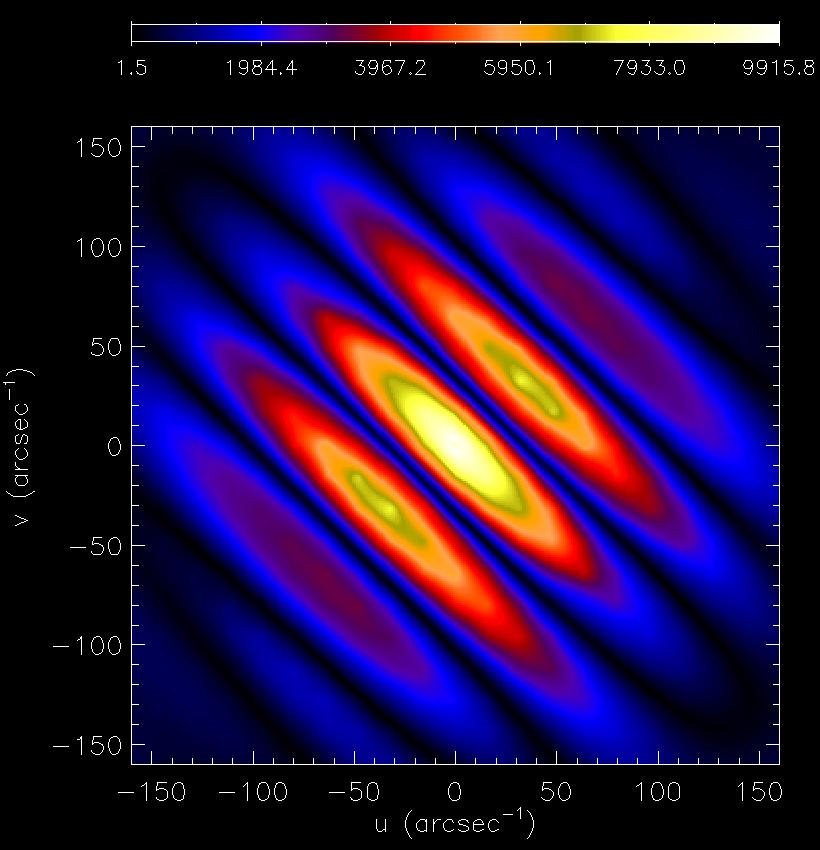}
  \includegraphics[scale=0.15]{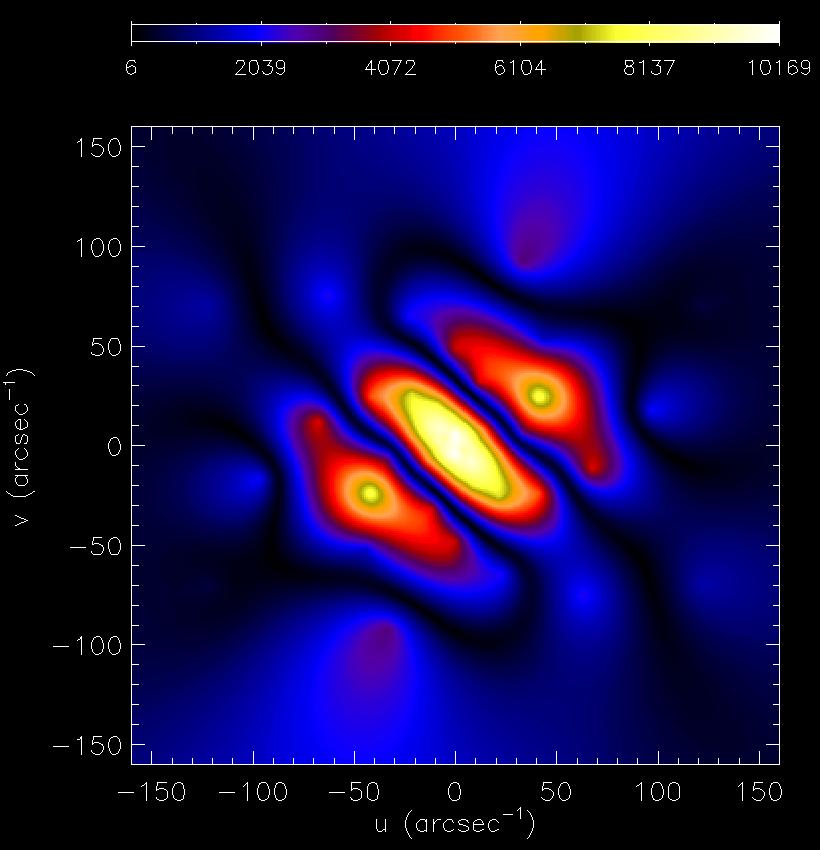}   
\caption{Top row, from left to right: ground truth and reconstructions of the synthetic double foot-point flare from {\em{STIX}} simulated visibilities, using Land-VSK and Land-RBF. The bottom row represents the visibility surfaces of the ground truth and the ones returned by the interpolation algorithms.}
		\label{fig4_new} 
\end{figure}

\begin{table}[ht]
\caption{Results concerning the reconstruction of the simulated double foot-point flare.}
\begin{center}
\begin{tabular}{lcccc}
\br
	 &$x_p$ (arcsec)&  $y_p$ (arcsec)&   FWHM (arcsec)  & FLUX (photon cm$^{-2}$ s$^{-1}$ $\times 10^3$)  \\
\mr
	& \multicolumn{4}{c}{First Peak} \\	 
\mr	 
	Simulated & -15.0 & -15.0 & 11.0 & 4.88 \\
	Land-VSK & -13.0 $\pm$ 0.2 & -13.1 $\pm$ 0.4 & 11.2 $\pm$ 0.2 & 3.91 $\pm$ 0.11 \\
    \smallskip
	Land-RBF & -11.8  $\pm$ 0.5 & -11.8  $\pm$ 0.4 & 13.9  $\pm$ 0.8 & 3.33  $\pm$ 0.16 \\
	\mr
	& \multicolumn{4}{c}{Second Peak}\\
       \mr
	 
	Simulated & 15.0 & 15.0 & 11.0 & 4.88 \\
	Land-VSK & 14.9  $\pm$ 0.2 & 15.0  $\pm$ 0.4 & 11.3  $\pm$ 0.2 & 3.91  $\pm$ 0.07 \\
	\smallskip
	Land-RBF      & 13.9  $\pm$ 0.5 & 13.6  $\pm$ 0.5 & 13.9 $\pm$ 0.5 & 3.34 $\pm$ 0.09 \\
	\br
\end{tabular}
\begin{tabular}{lcc}
	&  Total Flux (photon cm$^{-2}$ s$^{-1}$ $\times 10^3$) & RRMSE \\
	\mr
	Simulated & 10.00 &   \\
	Land-VSK  & 10.18 $\pm$ 0.23 & 0.26  $\pm$ 0.01\\
		\smallskip
	Land-RBF & 11.35  $\pm$ 0.30 & 0.45  $\pm$ 0.02\\
	\br
\end{tabular}
\label{tab1_new}
\end{center}
\end{table}

\begin{table}[ht]
\caption{Condition number and spectral ratio of the classical and VSK kernel matrices computed for the STIX visibilities, using the data samples of Figure \ref{fig4_new}.}
\begin{center}
\begin{tabular}{cccc}
\br
 cond$(\mK)$ & cond$(\mK^{\Psi})$ &  $S(\mK)$ & $S(\mK^{\Psi})$  \\
\mr 
8.410 {\rm e}+05 & 3.346 {\rm e}+05 & 1.002 & 1.006 \\
\br
    \end{tabular}
    \label{tabcond_new}
\end{center}    
\end{table}

\subsection{RHESSI real data}

As last example, we test our procedure on a flare observed by {\em{RHESSI}} on February 20 2002 during the time interval  11:06:02--11:06:24 UT. The energy range of the event is 50--70 keV. The results returned by Land-RBF and Land-VSK are illustrated in Figure \ref{fig5}. In this application with real observations, we also include a comparison with two other algorithms included the {\em{SSW}} tree: uv$\_$smooth, which is an interpolation/extrapolation algorithm where the interpolation step is realized by using a spline function \cite{Massone2009HARDXI}; and Clean \cite{schmahl2007analysis}, which is a deconvolution algorithm based on thresholding procedures. The image of the {\em{RHESSI}} double foot-point flare reconstructed via Land-VSK contains fewer artifacts than the ones produced by Land-RBF and uv$\_$smooth. Artefacts are negligible in the case of Clean reconstructions. However, Clean has two significant drawbacks with respect to Land-VSK: first, it is not completely user-independent, given that the last step of this iterative scheme requires the convolution with an idealized point spread function whose FWHM is manually chosen by the user via heuristic considerations; second, the $\chi^2$ values associated to the interpolation scheme is significantly smaller ($\chi^2=2.1$ for Land-VSK; $\chi^2=3.3$ for Clean).

Also for {\em{RHESSI}} data, in Table \ref{tabcond1}, we report the spectral ratios and condition numbers for the standard and VSK Mat\'ern kernel matrices. 

\begin{figure}
\centering
    \includegraphics[scale=0.13]{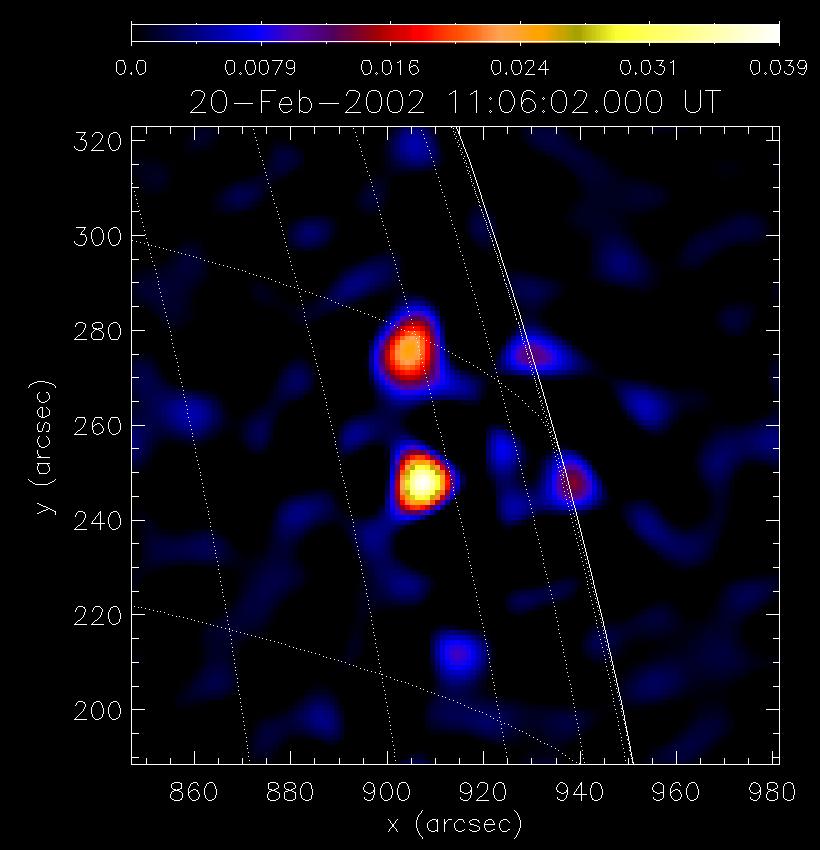} 
    \includegraphics[scale=0.13]{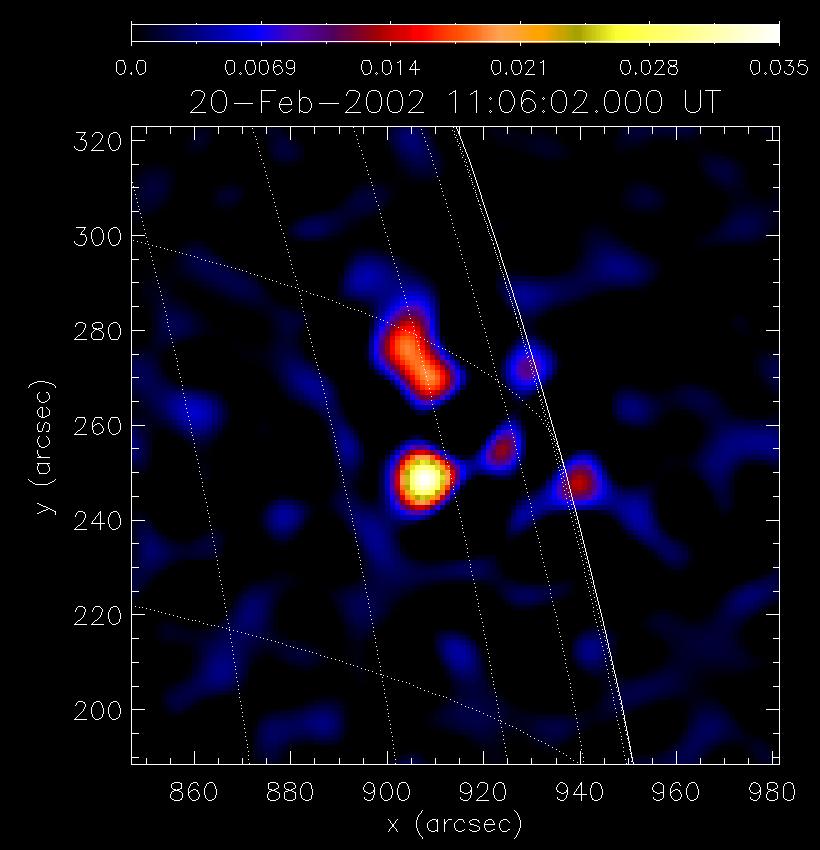} 
    \includegraphics[scale=0.13]{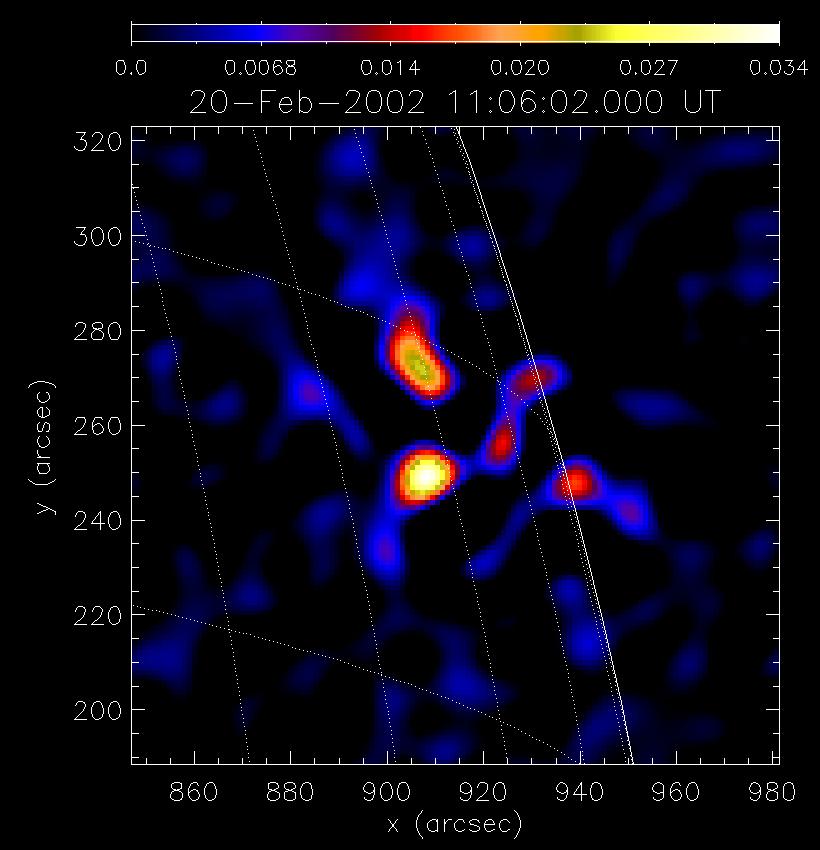}
   \includegraphics[scale=0.13]{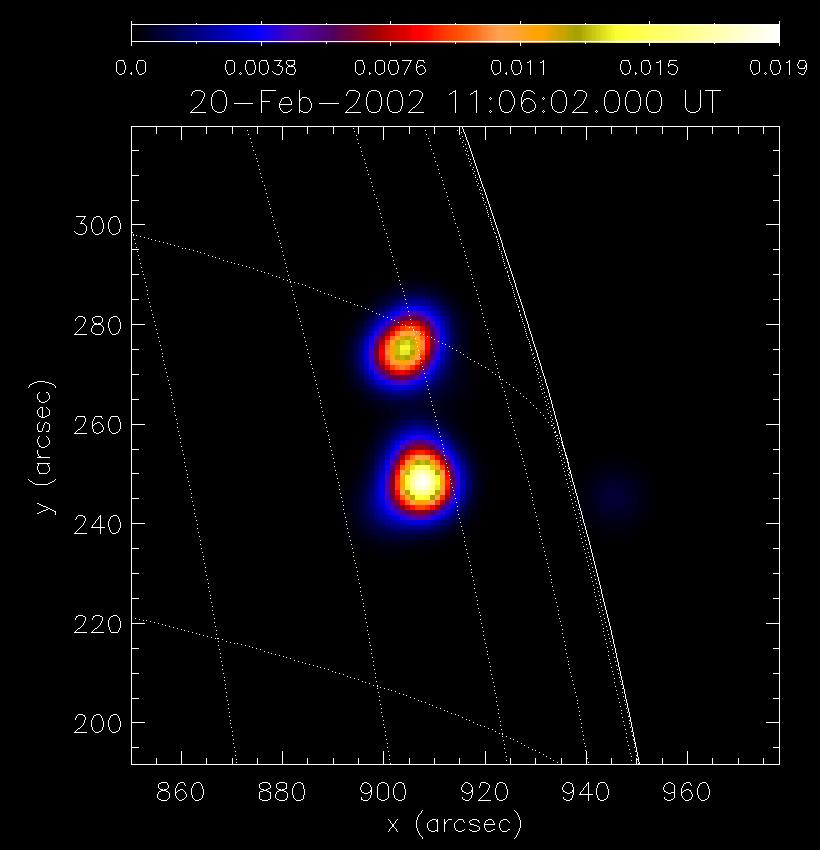}
    \caption{Reconstruction of the flare observed by {\em{RHESSI}} on February 20 in 2002. From left to right: Land-VSK,   Land-RBF, uv$\_$smooth and clean. 
}
    \label{fig5}
\end{figure}

\begin{table}[ht]
\caption{Condition number and spectral ratio of the classical and VSK kernel matrices computed for the {\em{RHESSI}} visibilities, using the data samples of Figure \ref{fig5}.}
\begin{center}
\begin{tabular}{cccc}
\br
cond$(\mK)$ & cond$(\mK^{\Psi})$ &  $S(\mK)$ & $S(\mK^{\Psi})$  \\
\mr
3.338 {\rm e}+05 & 1.753 {\rm e}+05 & 1.003 & 1.006 \\
\br
    \end{tabular}
    \label{tabcond1}
\end{center}    
\end{table}

\section{Comments and conclusions}
\label{conclusions}

The theoretical and data analysis results of this study show that the proposed Fourier inversion scheme, based on VSK interpolation and projected Landweber extrapolation, can be effectively used in many applications. As a case study, we focused on astronomical imaging and we pointed out that the use of VSKs seems to be the key ingredient for dealing with solar flares reconstructions, especially when we do not dispose of well-distributed data. The results on the {\em{STIX}} single and double foot-point flares point out that the classical kernel-based interpolation is sensitive when the peaks in the image plane are close to the boundary. Such shifts directly reflect on the visibility surfaces producing oscillations at the boundary of the visibility domain. Therefore, we need to use a data-driven interpolation via VSKs to make up for the lack of information. 

Finally, to numerically verify Propositions \ref{prop1} and \ref{prop2}, we computed the condition numbers and spectral ratios of the kernel matrices. The results are consistent with what theoretically proven. Future work concerns theoretical studies about the selection of kernel centres via greedy methods (refer e.g. to \cite{WH2013}) and the inclusion of Land-VSK in the {\em{SSW}} tree.

\section*{Acknowledgements}
This research has been accomplished within Rete ITaliana di Approssimazione (RITA). The authors acknowledge the financial contribution from the agreement ASI-INAF n.2018-16-HH.0 and the support of GNCS-IN$\delta$AM.

\section*{References}
\bibliographystyle{iopart-num}
\bibliography{bib_stix.bib}   
 
\end{document}